\documentclass[journal,twocolumn]{IEEEtran}%
\usepackage{lineno}
\usepackage[cmex10]{amsmath}
\usepackage{amscd,amsfonts,amsmath,amssymb,mathrsfs,pifont,stmaryrd,tipa}
\usepackage{array,cases,dsfont,graphicx,texdraw}
\usepackage{graphics}
\usepackage{epsfig}
\usepackage{cite}
\usepackage{mathbbold}
\usepackage{stfloats}
\usepackage{subfigure}
\usepackage{hyperref}
\usepackage{amsfonts}
\usepackage{epstopdf}
\usepackage{amssymb}
\usepackage{graphicx}%
\usepackage{xcolor}
\usepackage{hyperref} 
\usepackage{cleveref}
\usepackage{picins}

\newtheorem{theorem}{\bf{Theorem}}

\newtheorem{assumption}{\bf{Assumption}}

\newtheorem{lemma}{\bf{Lemma}}

\newtheorem{proposition}{Proposition}
\newenvironment{proof}{\noindent{\em Proof:\/}}{\hfill $\Box$\par}
\newtheorem{remark}{\bf{Remark}}

\allowdisplaybreaks
\begin{document}
	
	\title{\textcolor{black}{Distributed Nash Equilibrium Seeking for  Monotone Generalized Noncooperative Games by a Regularized Penalty Method}}
	\author{Chao Sun and Guoqiang Hu\thanks{This work was supported by Singapore Ministry of Education Academic Research Fund Tier 1 RG180/17(2017-T1-002-158). The results in this paper were partially presented at International Conference on Mechatronics and Control Engineering (ICMCE), 2017. C. Sun and G. Hu are with the School of
			Electrical and Electronic Engineering, Nanyang Technological University,
			Singapore 639798 (email: csun002@e.ntu.edu.sg, gqhu@ntu.edu.sg).}}
	\maketitle
	
	\begin{abstract}
		In this work, we study the distributed Nash equilibrium seeking problem for monotone generalized noncooperative games with set constraints and shared affine inequality constraints. A distributed regularized penalty method is proposed. The idea is to use a differentiable penalty function with a time-varying penalty parameter to deal with the inequality constraints. A time-varying regularization term is used to deal with the ill-poseness caused by the monotonicity assumption and the time-varying penalty term. The asymptotic convergence to the least-norm variational equilibrium of the game is proven. Numerical examples show the effectiveness and efficiency of the proposed algorithm.
	\end{abstract}
	

	\section{Introduction}

	The noncooperative game has attracted many interests in these years due to its wide applications in modelling of some economic, social and engineering problems. Nash equilibrium describes an equilibrium state that no player can benefit from individually changing its own action. In the noncooperative game studies, the Nash equilibrium seeking problem is one of the most important topic and there have been many methods proposed in the literature to handle this issue,  such as fictitious play, variational inequality, and extremum seeking. With the development of networked systems, there has been an increasing demand for a distributed algorithm design due to the advantages in failure control, resilience, and communication efficiency. Under such circumstances, the distributed Nash equilibrium seeking problem has arisen (\cite{Ye,Xie, de2019distributed,de2019continuous ,grammatico2015decentralized, zeng2019generalized,zhang2019distributed, liang2017distributed,gadjov2018passivity,pavel2019distributed,salehisadaghiani2016distributed,salehisadaghiani2018distributed,yi2019operator, yi2018distributed,tatarenko2018learning,lu2018distributed,franci2019damped,franci2019distributed,yu2017distributed,romano2019dynamic,yi2019asynchronous}), where each player can only interact with its neighboring players in the network.
	
	The monotonicity is an important concept in the studies of Nash equilibrium seeking, and it plays a similar role as that the convexity
	plays in optimization. Monotone games, strictly monotone games and strongly monotone games are three main types of games studied in the literature. Among the above monotonicity concepts, the following relationship holds:  strongly monotone $\rightarrow$ strictly monotone $\rightarrow$ monotone. In the current studies on distributed Nash equilibrium seeking, most of the existing works focus on strongly monotone games. For example, \cite{Ye,zhang2019distributed} proposed continuous-time algorithms and \cite{pavel2019distributed,yi2019operator} designed discrete-time algorithms. There are also some works investigating strictly monotone games, such as the discrete-time algorithm proposed in \cite{salehisadaghiani2016distributed} and the continuous-time algorithms proposed in  \cite{zeng2019generalized,de2019distributed,liang2017distributed,gadjov2018passivity}. In addition, \cite{kannan2012distributed,yin2011nash, grammatico2017proximal, belgioioso2017semi,9029659,yi2018distributed, franci2020distributed} studied monotone games where \cite{grammatico2017proximal, belgioioso2017semi,yi2018distributed, franci2020distributed} adopted operator theory based methods and \cite{9029659,kannan2012distributed, yin2011nash} developed regularization methods. All of these algorithms are discrete-time.
	
	In this work, we propose a novel approach to solve the distributed Nash equilibrium seeking for monotone generalized noncooperative games with set constraints and shared affine inequality constraints by using regularization and penalization. As is well-known, the Nash equilibrium seeking problem can be converted to a Variational Inequality (VI) problem. To solve the monotone VI with constraints, we first design a regularized and penalized VI with time-varying coefficients which has set constraints only. Its solution is proven to asymptotically converge to the original VI provided that the time-varying coefficients satisfy some  conditions. Then, a projected gradient based algorithm is proposed to solve the regularized and penalized VI, where leader-following consensus is used to achieve distributed information exchange.
	
	The contribution of this work can be summarized as follows:

	(a) The centralized regularization methods have been investigated in many works to solve monotone VIs (noncooperative games, equilibrium problems, convex optimization), e.g., the continuous-time algorithms in \cite{cominetti2008strong, boct2020inducing} and the discrete-time algorithms in \cite{facchinei2007finite}\cite{koshal2010single}. However, most of the methods don't consider inequality constraints. In this work, we add a time-varying penalty term into the regularized dynamics to deal with the constraints. According to the authors' knowledge, this is the first work to propose such a dynamical system with both regularized terms and penalized terms and analyze the convergence. The problem transformation from a constrained  monotone VI to a regularized and penalized VI has been preliminarily studied in \cite{konnov2015regularized}, however, \cite{konnov2015regularized} didn't consider its algorithm solution and convergence. Furthermore, the assumptions in our work are indeed independent of the assumptions in \cite{konnov2015regularized} (i.e., C2 and C2'). In addition, in \cite{ attouch2010asymptotic, attouch2011coupling}, continuous-time dynamical systems without regularization were proposed which can be used to solve monotone VI with inequality constraints. However, only ergodic convergence was proven.
	 
	 	 (b) The contribution to distributed Nash equilibrium seeking: to the best of the authors' knowledge, this  is the first work to propose a distributed continuous-time algorithm for monotone (generalized) noncooperative games. In addition, this is the first work to apply the regularized penalty method to solve distributed Nash equilibrium seeking problems. Compared with the operator theory based methods in \cite{grammatico2017proximal, belgioioso2017semi,yi2018distributed, franci2020distributed}, the proposed algorithm is simpler and requires less communication since no multiplier information exchange is needed. Benefiting from the adopted regularization method, the obtained solution is the least-norm solution, which may have potential applications in some practical problems. For example, in a  convexly constrained linear inverse problems \cite{sabharwal1998convexly}, which to some extent can be viewed as a special case of this work, the least-norm solution represents the least possible
	 	 signal energy. Among the regularization based methods,
	 	 \cite{9029659,kannan2012distributed} used pure regularization for set constrained games, \cite{yin2011nash} designed a primal-dual regularization algorithm, while we propose a regularized penalty algorithm, which does not require the evolution of multiplier. In detail, the algorithms in \cite{ yin2011nash, kannan2012distributed}  use full-decision information and \cite{9029659} investigated box constraints only. In addition, our algorithm doesn't require the global Lipschitz assumption and the linear growth assumption. Overall, our algorithm is novel in both continuous-time and discrete-time algorithms.
	 	 
	 	 	\color{black}
	 	 
	 	The remainder of this paper is structured as follows: In Section \ref{S2}, notations and preliminary knowledge are given. In Section  \ref{S3}, the problem formulation is described. In Section \ref{S4}, the regularized penalized VI is proposed and the relationship between its solution with the least-norm variational equilibrium of the original game is studied. Then, a decentralized algorithm using full-decision information and a  distributed algorithm using partial-decision information  are designed, respectively.  Asymptotic convergence to the  least-norm variational equilibrium is proven for both algorithms. In Section \ref{S5}, numerical examples are provided to verify the effectiveness and efficiency of the proposed algorithms.
	 	
				\color{black}

	\section{Notations and Preliminaries\label{S2}}

	\subsection{Notations\label{Notation}}
	
	Throughout this paper, $\mathbb{R}$, $\mathbb{R}^n$ and   $\mathbb{R}^{n\times n}$ denote the real number set, the $n$-dimensional real vector set, and the $n\times n$ dimensional real matrix, respectively. $||\cdot||$ and $|\cdot|$ represent the $2$-norm and the absolute value, respectively. Let $\mathcal{V}=\left\{  1,...,N\right\}$. Then, $[z_i]_{i\in\mathcal{V}}$ denotes a column vector $[z_1^T,\cdots,z_N^T]^T$. For a  vector $z\in\mathbb{R}^N$, $[z]_i$ is the $i$-th element of $z$. For a set $Q$ and a variable $x$, $\mathcal{P}_{Q}[x]=\text{argmin}_{y\in Q}||x-y||^2$ represents the projection of  $x$ onto $Q$. The notation $x\overset{a.e.}=y$ represents $x$ is equal to $y$ at almost everywhere. In this work, $0$ represents an all-zero vector with an appropriate dimension or the real number zero. $\mathbf{1}_N$ represents an $N$-dimensional column vector with all elements being 1. For a vector $g$, $g\leq 0 $ means that each element of $g$ is less than or equal to zero.
		
	\subsection{Preliminaries \label{Graph theory}}
	
	Let $\mathcal{G}=\left\{  \mathcal{V},\mathcal{E}\right\}  $ denote an
	undirected graph, where $\mathcal{V=}\left\{  1,...,N\right\}  $ indicates the
	vertex set and $\mathcal{E\subset V\times V}$ indicates the edge set.
	$\mathcal{N}_{i}=\left\{  j\in\mathcal{V\mid}(j,i)\in\mathcal{E}\right\}  $
	denotes the neighborhood set of vertex $i$.  Path $\mathcal{P}$ between $v_{0}$ and $v_{k}$\ is the
	sequence $\left\{  v_{0},...,v_{k}\right\}  $ where $(v_{i-1},v_{i}%
	)\in\mathcal{E}$ for $i=1,...,k$ and the vertices are distinct. The number $k$
	is defined as the length of path $\mathcal{P}$. Graph $\mathcal{G}$\ is
	connected if for any two vertices, there is a path in $\mathcal{G}$. A matrix
	$A_d=\left[  a_{ij}\right]  \in%
	\mathbb{R}
	^{N\times N}$ denotes the adjacency matrix of $\mathcal{G}$, where $a_{ij}>0$
	if and only if $(j,i)\in\mathcal{E}$ else $a_{ij}=0$. In this paper, we
	suppose $a_{ii}=0.$ A matrix $L\triangleq D-A_d\in%
	\mathbb{R}
	^{N\times N}$ is called the Laplacian matrix of $\mathcal{G}$, where
	$D=\text{diag} { \{d_{ii}\}} \in%
	\mathbb{R}
	^{N\times N}$ is a diagonal matrix with $d_{ii}=\sum\nolimits_{j=1}^{N}a_{ij}
	$ \cite{RenTAC05}.

	 A function $f:%
	\mathbb{R}
	^{n}\rightarrow%
	\mathbb{R}
	$ is locally Lipschitz at $x\in%
	\mathbb{R}
	^{n}$ if there exists a neighborhood $\Omega$ of $x$ and $C \geq 0 $ such that $|f(y)-f(z)|\leq C \left\Vert y-z\right\Vert $ for
	$y,z\in\Omega.$ $f$ is locally Lipschitz on $%
	\mathbb{R}
	^{n}$ if it is locally Lipschitz for all $x\in%
	\mathbb{R}
	^{n}.$ $f$ is globally Lipschitz on $%
	\mathbb{R}
	^{n}$ if for $y,z\in%
	\mathbb{R}
	^{n},$ there exists $C\geq0$ such that $|f(y)-f(z)|\leq C\left\Vert y-z\right\Vert $ \cite{Clark,Jorge}.
	
	A single-valued map $F: \mathbb{R}^n\rightarrow \mathbb{R}^n$ is monotone if for all $x,y\in \mathbb{R}^n$, $(F(x)-F(y))^T(x-y)\geq 0$. It is strictly monotone if the inequality is strict when $x\neq y$. $F$ is strongly monotone if there exists a positive constant $m$ such that $(F(x)-F(y))^T(x-y)\geq m||x-y||^2$ for all $x,y\in \mathbb{R}^n$.
	
	A set $\Omega \subseteq \mathbb{R}^n$ is closed if and only if the limit of every Cauchy sequence (or convergent sequence) contained in $\Omega$ is also an element of $\Omega$.
	
	Let $x_k\in\mathbb{R}^n$, $k=1,\cdots, \infty$ be a sequence of real vectors, $\bar{x}$ is called an accumulation point of sequence $x_k$ if there exists a subsequence $x_{k_s}$ of $x_k$ such that  $\lim_{k_s \to \infty} x_{k_s} = \bar{x}$.   For a bounded sequence $x_k\in\mathbb{R}^n$, it contains at least one accumulation point. Furthermore, the accumulation point is unique if and only if the sequence is convergent.
	
	The following lemma will be used in the subsequent stability analysis.

	\begin{lemma} (Theorem 1.5.5 of \cite{facchinei2007finite}) Let $\Omega$ be a  nonempty, closed and convex subset of $\mathbb{R}^N$. Then, the following conclusions hold: 
		
		 (a) for any $x, y\in\mathbb{R}^N$,
	$(P_\Omega[x]-P_\Omega[y])^T(x-y)\geq||P_\Omega[x]-P_\Omega[y]||^2$ and $||P_\Omega[x]-P_\Omega[y]||\leq ||x-y||$; 
	
	(b) $(y-P_\Omega[x])^T(P_\Omega[x]-x)\geq 0$ for any $y\in\Omega$. \label{projector}
	\end{lemma}

	\section{Problem Formulation\label{S3}}
	
	Consider a set of players $\mathcal{V}\triangleq\{1,\cdots,N\}$ defined in the graph $\mathcal{G}$, where the
	action of player $i$ is denoted as $x_{i}\in%
	\mathbb{R}
	$. Each player $i$ minimizes its own objective function $f_{i}(x):%
	\mathbb{R}
	^{N}\rightarrow%
	\mathbb{R}
	$ where $x=\left[  x_{1},\cdots,x_{N}\right]  ^{T}\in%
	\mathbb{R}
	^{N}$ is the action vector of $N$ players. The results in this work can be easily extended to the case when $x_i\in\mathbb{R}^{m_i}$ with $m_i $ being a positive integer. For the notational simplicity, we take  $m_i=1$.

	 If agent $j$ is not a neighbor of agent $i,$ then player $i$ has
	no access to player $j$'s action directly. Otherwise, player $i$ can get the
	information of player $j$ via a connected undirected graph topology. Our
	objective is to design a distributed Nash equilibrium seeking law for the
	players such that their actions converge to a Nash equilibrium of the game.
	
	Denote $x_{-i}=\left[  x_{1},\cdots,x_{i-1},x_{i+1},\cdots,x_{N}\right]
	^{T}\in%
	\mathbb{R}
	^{N-1}.$  Player $i$' problem can be
	described as%
	\begin{align}
	&\min_{x_{i}\in\Omega_i}\text{ }f_{i}(x_{i},x_{-i}),\notag \\
	&\text{ such that } g(x_i,x_{-i})\leq 0, \label{P1}%
	\end{align}
	where $\Omega_i$ denotes the local strategy space of player $i$, and $g(x_i,x_{-i})=Ax-b\in\mathbb{R}^n$ is a shared inequality constraint function with $n$ being a positive integer,  $A\in\mathbb{R}^{n\times N}$ and $b\in\mathbb{R}^{n}$. 
	
	 In the following development, we simply the notations of $f_{i}(x_{i},x_{-i})$ and $g(x_i,x_{-i})$ to $f_{i}(x)$ and $g(x)$.
	 
	 The following mild assumptions are used in the subsequent analysis.
	 
		\begin{assumption}
		\label{fassumption}	(a)  The objective function $f_{i}(x)$ is twicely continuously differentiable in $x$; 	(b)  $f_{i}(x)$ is convex in $x_i\in\mathbb{R}$ for any $x_{-i}\in\mathbb{R}^{N-1}$.
	\end{assumption}
	
	\begin{assumption}
		\label{monot}
	The map $ F(x)\triangleq \left[ \nabla_{x_{i}} f_{i}(x)\right]_{i\in\mathcal{V}}$ is monotone in $\mathbb{R}^{N}$.  
	\end{assumption}
	
	\begin{assumption}
		\label{gassumption2} (a) $\Omega_i$ is a nonempty, convex and compact set; (b) The set $Q\triangleq\{x\in\Omega | g(x)\leq 0\}$ satisfies the Slater's condition, where $\Omega\triangleq \Omega_1\times\cdots\times\Omega_N $.
	\end{assumption}

	\begin{assumption}
		The communication graph is undirected and connected. \label{graph}
	\end{assumption}

	\section{Distributed algorithm for monotone games} \label{S4}

	\begin{figure}
		\centering
		\hspace*{-2.5cm}\includegraphics[width=1.8\linewidth]{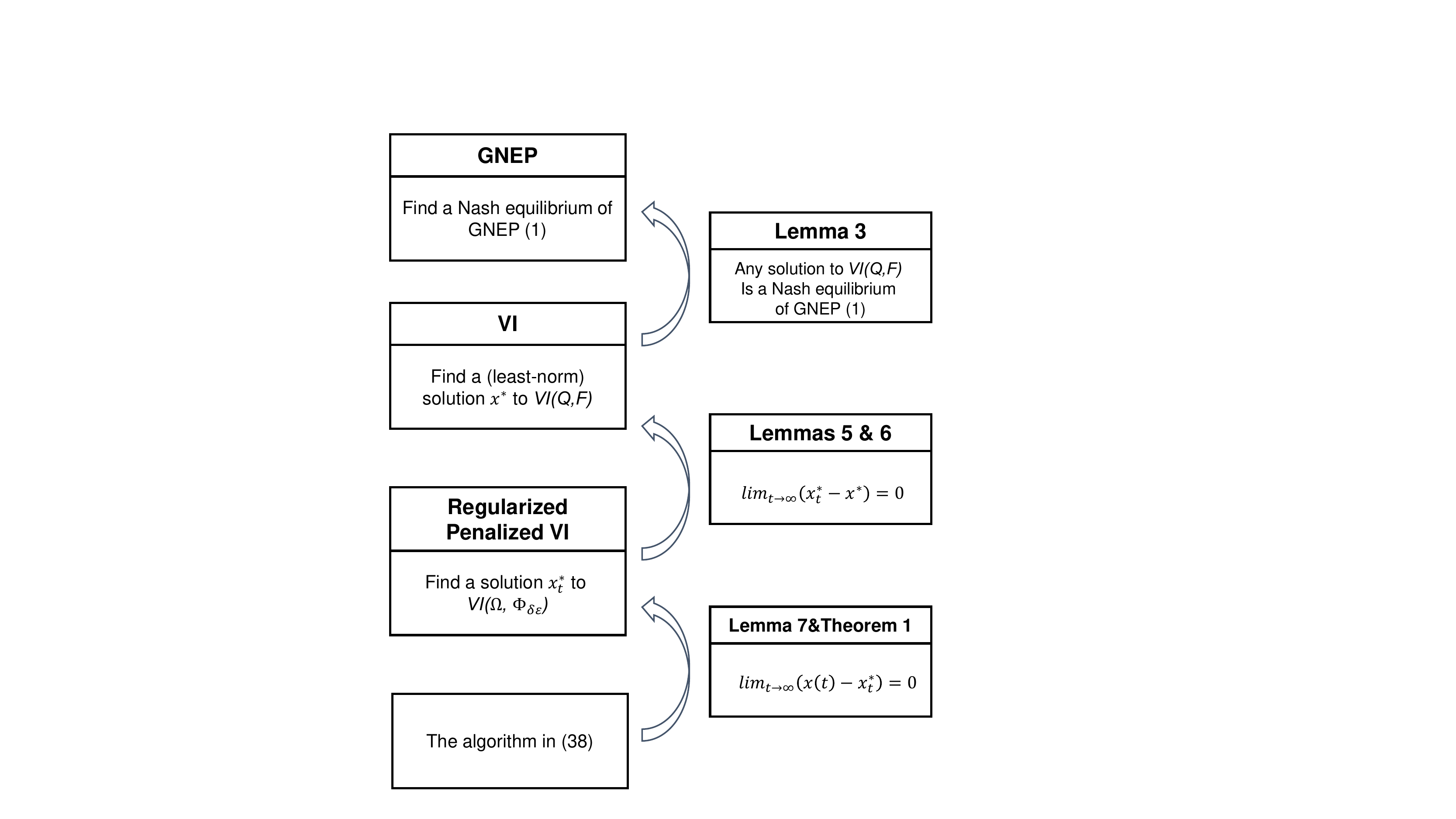}
		\caption{\textcolor{black}{Framework of the convergence analysis for the algorithm in (\ref{dynamics01}).}}
		\label{fig:algorithm2}
	\end{figure}

	\color{black}
   In this work, we develop a distributed regularized penalty method to solve GNEP (\ref{P1}). The proposed method is motivated by the regularized penalty method in \cite{ konnov2015regularized} for the problem transformation of a constrained equilibrium problem (EP) to a regularized and penalized EP, the differential equation method to find a zero point of a regularized maximal monotone operator in \cite{cominetti2008strong}, the projected gradient dynamics in \cite{cavazzuti2002nash} for Nash equilibrium seeking and the consensus based approach to solve a distributed Nash equilibrium seeking problem in \cite{Ye}. We first provide a decentralized algorithm using full-decision information, and then present the distributed algorithm using partial-decision information. The framework of the analysis for a full-decision information algorithm is described in Fig. \ref{fig:algorithm2}. First of all, we analyze the relationship between the solution of a regularized penalized problem with that of the original problem. 

   	\color{black}
		
	\subsection{Variational Inequality and Equilibrium Problems}
	In this section, we introduce several fundamental concepts in variational inequality and equilibrium problems.
	
Let $Q$ be a nonempty, convex and compact set and $F(x): Q\rightarrow \mathbb{R}^N$ be a continuous single-valued mapping, the VI problem, denoted by VI($Q, F$), is to find an $x\in Q$ such that 
\begin{equation}
(y-x)^TF(x)\geq 0, \forall y\in Q. \label{sol}
\end{equation}

 The Dual Variational Inequality problem (DVI), denoted by DVI($Q, F$), is to find an $x\in Q$ such that 
\begin{equation}
(y-x)^TF(y)\geq 0, \forall y\in Q. \label{dvi}
\end{equation}

The EP, denoted by EP($Q, \Gamma$), is defined as follows: to find an $x\in Q$ such that 
\begin{equation}
\Gamma(x,y)\geq 0, \forall y\in Q.
\end{equation}
where $\Gamma(x,y): Q \times Q\rightarrow \mathbb{R}$ is a bifunction with respect to $x$ and $y$.

The  Dual Equilibrium Problem (DEP), denoted by DEP($Q, \Gamma$), is defined as follows:  to find an $x\in Q$ such that 
\begin{equation}
\Gamma(y,x)\leq 0, \forall y\in Q.
\end{equation}

It is obvious that  VI($Q, F$) is equivalent to EP($Q, \Gamma$) with the  bifunction $\Gamma(x,y)=(y-x)^TF(x)$. In addition, DVI($Q, F$) is equivalent to DEP($Q, \Gamma$) where $\Gamma(y,x)=(x-y)^TF(y)$.

Furthermore, since $\Omega$ and the set $\{g(x)\leq 0\}$ are closed and convex under Assumption  \ref{gassumption2}, $Q$ is a closed and convex set. Then, the following conclusion holds:

\begin{lemma} (Lemma 2.1 of \cite{cottle1992pseudo})
		Under Assumptions 	\ref{fassumption}--\ref{gassumption2}, any solution to VI($Q, F$)  is a solution to  DVI($Q, F$) and vice versa. It follows that any solution to  EP($Q, \Gamma$) is a solution to DEP($Q, \Gamma$) and vice versa. \label{r4}
\end{lemma}

\subsection{Regularized Penalty Method for VI}

It is well-known that VI can be used to solve a Nash equilibrium seeking problem. The following lemma shows the relationship between a Nash equilibrium of GNEP (\ref{P1}) and a solution of VI($Q, F$).

\begin{lemma} (Theorem 5 of \cite{facchinei2007generalized})
	Under Assumptions \ref{fassumption} and \ref{gassumption2}, any solution to  VI($Q, F$) is a Nash equilibrium of GNEP (\ref{P1}), called a variational equilibrium.
\end{lemma}

First, we  transform the constrained monotone VI in (\ref{dvi}) to a strongly monotone VI with the set constraint $\Omega$ only, by partial penalization and regularization.

\textbf{Regularized VI: VI(Q, $\mathbf{\Phi_{\delta}}$)}

The regularized map $\Phi_{\delta}(x)$ is defined as
\begin{equation}
\Phi_{\delta}(x)=F(x)+\delta(t)x,
\end{equation}
where $\delta(t)$ is a  time-varying, positive and smooth parameter with $\delta(0)<\infty$, $\delta(t)x$ is a regularization term that is used due to the monotonicity assumption.

\textbf{Regularized Penalized VI: VI($\mathbf{\Omega}$, $\mathbf{\Phi_{\delta\varepsilon}}$)}

To deal with the inequality constraints, we utilize a quadratic penalty function defined by
$P(x)=
\sum_{k=1}^nP_k(x)$ where
\begin{align}
P_k(x)=
\begin{cases}
0, & \text{if  } g_k(x) \leq 0,\\
{(g_k(x))^2},  & \text{if else,} 
\end{cases}  \label{penalty}
\end{align}
and $g_k(x)$ is the $k$-th row of  $g(x)$. $P(x)$ is once continuously differentiable and satisfies
\begin{equation}
P(x)=0 \text{ if } g(x) \leq 0, P(x)>0 \text{ if else}.
\end{equation}

Define a new map $\Phi_{\delta\varepsilon}(x)$ as follows:
\begin{equation}
\Phi_{\delta\varepsilon}(x)=F(x)+\delta(t)x+\varepsilon(t)\nabla_x P(x), \label{phiy}
\end{equation}
where $\varepsilon(t)$ is a time-varying, positive and smooth parameter with $\varepsilon(0)<\infty$.

The regularized penalized VI investigated in this work is defined by   VI($\Omega, \Phi_{\delta\varepsilon}$):
to find an $x\in \Omega$ such that 
\begin{equation}
(y-x)^T\Phi_{\delta\varepsilon}(x)\geq 0, \forall y\in \Omega. \label{sol2}
\end{equation}

We aim to seek a Nash equilibrium of GNEP (\ref{P1}) by solving  VI($\Omega, \Phi_{\delta\varepsilon}$). To this end, we analyze the relationship between the solutions to VI($Q, F$) and  VI($\Omega, \Phi_{\delta\varepsilon}$) over time.

\begin{remark} Most of the existing methods to solve VI($Q, F$) requires at least strict monotonicity of $F$. Methods to deal with a monotone $F$ include the extragradient method, the  hyperplane projection method, the regularization method,  the proximal point method, and the splitting method \cite{facchinei2007finite}, etc. The proposed regularized penalty method is motivated by the regularization method. The difference with a general regularization method is that a time-varying penalty function is used to deal with inequality constraints. In addition, we prefer   a differential equation solution different from the discrete-time solutions in most of the existing regularization methods.

Note that when the inequality constraint is not considered (i.e., there is no penalty term $\varepsilon(t)\nabla_x P(x)$), the relationship between the regularized VI and the original problem,  and the corresponding algorithms  were investigated in  \cite{facchinei2007finite}, etc. However, when the penalty term is involved, the existing analysis and conclusions do not hold (e.g., Theorem 12.2.3 of  \cite{facchinei2007finite}). Thus, here we generalize some key conclusions from the unpenalized case to the penalized case.

\end{remark} 

\color{black}
Define the following bifunction
\begin{align}
\Gamma_{\delta\varepsilon}(x,y)=(y-x)^TF(x)+\delta(t)(||y||^2-||x||^2)\notag\\
+\varepsilon(t)
(P(y)-P(x)).
\end{align}

According to Proposition 2.2  of \cite{gwinner1981penalty}, the following conclusion holds:

\begin{lemma} 
		Under Assumptions \ref{fassumption}--\ref{gassumption2}, at any time $t$, a vector $x$ belongs to SOL($\Omega, \Phi_{\delta\varepsilon}$) if and only if it solves EP($\Omega, \Gamma_{\delta\varepsilon}$).
\end{lemma}

Furthermore, we can get the following conclusion:

\begin{lemma}
	 Let $\delta_k$ and $\varepsilon_k$ be two arbitrary and positive sequences, and $x_{\delta\varepsilon}^k$ be a solution to EP($\Omega, \Gamma_{\delta_k\varepsilon_k}$). Then, under Assumptions \ref{fassumption}--\ref{gassumption2}, each sequence $x_{\delta\varepsilon}^k$ with $\delta_k\rightarrow 0$ and $\varepsilon_k\rightarrow \infty$ has at least one accumulation point, and all these accumulation  points are solutions to SOL($Q, F$). \label{l7} 
\end{lemma}

\begin{proof}
	 According to the strong monotonicity of $\Phi_{\delta\varepsilon}$, $x_{\delta\varepsilon}^k$ exists and is unique for each $k$ (Theorem 2.3.3 of \cite{facchinei2007finite}). Since $\Omega$ is bounded, the sequence $x_{\delta\varepsilon}^k$ is bounded and thus has at least one accumulation point. Let $\bar{x}$ be an arbitrary accumulation point and $x_{\delta\varepsilon}^{k_s}$ is a subsequence of $x_{\delta\varepsilon}^k$ that converges to $\bar{x}$, i.e., $\lim_{k_s\rightarrow\infty}x_{\delta\varepsilon}^{k_s}=\bar{x}$. Since we consider a finite-dimensional space $\mathbb{R}^N$,   weak convergence and strong convergence are equivalent. According to the closeness of $\Omega$,  $\bar{x}\in\Omega$ (Section \ref{S2}).
	 
	  According to the definition of EP, for any $y\in\Omega$,
	 \begin{align} 
	 0\leq P(x_{\delta\varepsilon}^{k_s})\leq \varepsilon_{k_s}^{-1} (y-x_{\delta\varepsilon}^{k_s})^TF(x_{\delta\varepsilon}^{k_s})+P(y)\notag\\
	 +\varepsilon_{k_s}^{-1}\delta_{k_s}(||y||^2-||x_{\delta\varepsilon}^{k_s}||^2). \label{r1}
	 \end{align} 
	 
	 According to the continuity of $P$, $P(\bar{x})= \lim_{k_s\rightarrow\infty}P(x_{\delta\varepsilon}^{k_s})$. Furthermore,  for any $y\in Q$,
	 \begin{equation} 
	 \lim_{k_s\rightarrow\infty}( \varepsilon_{k_s}^{-1} (y-x_{\delta\varepsilon}^{k_s})^TF(x_{\delta\varepsilon}^{k_s})+\varepsilon_{k_s}^{-1}\delta_{k_s}(||y||^2-||x_{\delta\varepsilon}^{k_s}||^2))=0. \notag
	 \end{equation} 
	 
	According to the Squeeze theorem, $P(\bar{x})=0$ and $\bar{x}\in Q.$
	
	In addition, since
	 \begin{align} 
0\leq \varepsilon_{k_s}P(x_{\delta\varepsilon}^{k_s}) \leq(\bar{x}-x_{\delta\varepsilon}^{k_s})^TF(x_{\delta\varepsilon}^{k_s})\notag\\
+\delta_{k_s}(||\bar{x}||^2-||x_{\delta\varepsilon}^{k_s}||^2), \label{x}
	\end{align} 
	and 
	\begin{equation}
\lim_{k_s\rightarrow\infty}( (\bar{x}-x_{\delta\varepsilon}^{k_s})^TF(x_{\delta\varepsilon}^{k_s})
+\delta_{k_s}(||\bar{x}||^2-||x_{\delta\varepsilon}^{k_s}||^2))=0,
	\end{equation}
	we have $\lim_{k_s\rightarrow\infty}\varepsilon_{k_s}P(x_{\delta\varepsilon}^{k_s})=0$ based on the Squeeze theorem, 
	which implies that for any $y\in Q$,
	 \begin{eqnarray}
& &	(y-\bar{x})^TF(\bar{x})=\lim_{k_s\rightarrow\infty}(y-x_{\delta\varepsilon}^{k_s})^TF(x_{\delta\varepsilon}^{k_s})\notag\\
& &	\geq \lim_{k_s\rightarrow\infty}(\varepsilon_{k_s}P(x_{\delta\varepsilon}^{k_s})-\delta_{k_s}(||y||^2-||x_{\delta\varepsilon}^{k_s}||^2))=0.\notag 
 \end{eqnarray} 
 
 Thus, $\bar{x}$ is a solution to SOL($Q, F$).
\end{proof}
	


%
%
%
%
%
%

\begin{remark}	
Lemma \ref{l7} is a special case of Theorem 4.1 in \cite{konnov2015regularized} for a bounded constraint set $\Omega$ and a monotone mapping. For the completeness, we adapt the proof so that it is compatible with the studied problem in this work.
\end{remark}

In the following, for notational convenience, we denote  the solution to VI($\Omega, \Phi_{\delta\varepsilon}$) at time $t$ by $x_t^*$. The relationship between $x_t^*$ and the least-norm solution to VI($Q, F$) can be described as follows.

\color{black}
\begin{lemma} 
		Under Assumptions \ref{fassumption}--\ref{gassumption2}, if $\lim_{t\rightarrow\infty} \delta^2(t)\varepsilon(t)$ $=\infty$,  $\lim_{t\rightarrow\infty} \delta(t)=0$, and $\lim_{t\rightarrow\infty} \varepsilon(t)=\infty$, then $\lim_{t\rightarrow\infty}x_t^*=x^*$, where $x^*$ is the least-norm solution to VI($Q, F$). \label{relationship}
\end{lemma}

\begin{proof} 
	According to Page 1128 of \cite{facchinei2007finite}, VI($Q, F$) has a unique least-norm solution.

According to the definition, for an arbitrary sequence $x_{\delta\varepsilon}^k$, we have for any $y\in\Omega$,
\begin{align} 
0\leq P(x_{\delta\varepsilon}^{k})\leq \varepsilon_{k}^{-1} (y-x_{\delta\varepsilon}^{k})^TF(x_{\delta\varepsilon}^{k})\notag\\
+P(y)+\varepsilon_{k}^{-1}\delta_{k}(||y||^2-||x_{\delta\varepsilon}^{k}||^2). \label{r2}
\end{align} 

Since $\delta_{k}(||y||^2-||x_{\delta\varepsilon}^{k}||^2)$ is uniformly bounded, there exists a positive constant $C_1$ such that $P(x_{\delta\varepsilon}^{k})\leq \frac{C_1}{\varepsilon_{k}}$ by selecting $y\in Q$. According to (\ref{penalty}), 
$[Ax_{\delta\varepsilon}^{k}-b]_i\leq \sqrt{\frac{C_1}{\varepsilon_{k}}}$, where $i\in\{1,\cdots, n\}$. 

Define an auxiliary variable $x^\dagger=x+v_{k}(x_\text{int}-x)$ (and thus $x=\frac{x^\dagger-v_{k}x_\text{int}}{1-v_{k}}$), where $x_\text{int}$ is a strict interior of $Q$ (i.e., $x\in\Omega $ and $Ax_\text{int}-b<0$) and its existence is guaranteed by the Slater condition, and $v_{k}$ is a positive sequence defined by
\begin{eqnarray}
0<v_{k}=\frac{\sqrt{\frac{C_1}{\varepsilon_{k}}}}{\sqrt{\frac{C_1}{\varepsilon_{k}}}+\min_{i=1,\cdots, n}{|[Ax_{\text{int}}-b]_i|}}<1.
\end{eqnarray}

Thus, for any $x\in\Omega$ and $ Ax-b\leq \sqrt{\frac{C_1}{\varepsilon_{k}}}\mathbf{1}_n$, we have $x^\dagger\in\Omega$, and 
\begin{eqnarray}
Ax^\dagger-b&=&(1-v_{k})Ax+v_{k}Ax_\text{int}-b\notag \\
&=&(1-v_{k})(Ax-b)+v_{k}(Ax_\text{int}-b)\notag \\
&\leq&(1-v_{k})\sqrt{\frac{C_1}{\varepsilon_{k}}}\mathbf{1}_n+v_{k}(Ax_\text{int}-b) \notag \\
&\leq& v_{k}(\min_{i=1,\cdots, n}{|[Ax_{\text{int}}-b]_i|}\mathbf{1}_n\notag \\
& &+Ax_\text{int}-b) \leq 0,
\end{eqnarray}
which implies that $x^\dagger\in Q$ if $x\in\Omega$ and $ Ax-b\leq \sqrt{\frac{C_1}{\varepsilon_{k}}}\mathbf{1}_n$.

Denote $x^{k}_Q=\mathcal{P}_Q[x_{\delta\varepsilon}^{k}]=\text{argmin}_{y\in Q}||y-x_{\delta\varepsilon}^{k}||^2$. Then, 
\begin{eqnarray}
||x_{\delta\varepsilon}^{k}-x^{k}_Q||^2&\leq& ||x_{\delta\varepsilon}^{k}-x^\dagger|_{x=x_{\delta\varepsilon}^{k}}||^2\notag \\
&\leq&\max_{x\in\Omega, Ax-b\leq \sqrt{\frac{C_1}{\varepsilon_{k}}}\mathbf{1}_n}||x-x^\dagger||^2\notag \\
&\leq& \max_{x^\dagger\in Q}||\frac{x^\dagger-v_{k}x_\text{int}}{1-v_{k}}-x^\dagger||^2 \notag \\
&=&(\frac{v_{k}}{1-v_{k}})^2\max_{x^\dagger\in Q}||x^\dagger-x_\text{int}||^2\notag \\
&\leq &\frac{C_2}{\varepsilon_{k}}, \label{im1}
\end{eqnarray}
where $C_2$ is some positive constant.

According to (\ref{r2}),
\begin{align} 
0 \geq& (x_{\delta\varepsilon}^{k}-x^*)^TF(x_{\delta\varepsilon}^{k})
+\delta_{k}(||x_{\delta\varepsilon}^{k}||^2-||x^*||^2) \notag \\
= & (x^{k}_Q-x^*)^TF(x_{\delta\varepsilon}^{k})+(x_{\delta\varepsilon}^{k}-x^{k}_Q)^TF(x_{\delta\varepsilon}^{k})\notag\\
 &+\delta_{k}(||x_{\delta\varepsilon}^{k}||^2-||x^*||^2)  \notag \\
=& (x^{k}_Q-x^*)^T(F(x_{\delta\varepsilon}^{k})-F(x^{k}_Q))+(x^{k}_Q-x^*)^TF(x^{k}_Q)\notag\\
 &+(x_{\delta\varepsilon}^{k}-x^{k}_Q)^TF(x_{\delta\varepsilon}^{k})+\delta_{k}(||x_{\delta\varepsilon}^{k}||^2-||x^*||^2)  \notag \\
\geq &-C_3||x_{\delta\varepsilon}^{k}- x^{k}_Q||+\delta_{k}(||x_{\delta\varepsilon}^{k}||^2-||x^*||^2),
\end{align} 
where in the last inequality we used the local Lipschitz continuity of $F$, (\ref{dvi}) for $y=x^{k}_Q$, and the boundedness of $F(x_{\delta\varepsilon}^{k})$ with $C_3$ being some positive constant.

Then, by (\ref{im1}),
\begin{equation} 
||x_{\delta\varepsilon}^{k}||^2-||x^*||^2\leq \frac{C_3}{\delta_{k}}\sqrt{\frac{C_2}{\varepsilon_{k}}}. \label{lim1}
\end{equation} 

Thus, for any convergent subsequence $x_{\delta\varepsilon}^{k_s}$ with an accumulation point $\bar{x}$,
\begin{equation} 
||x_{\delta\varepsilon}^{k_s}||^2-||x^*||^2\leq \frac{C_3}{\delta_{k_s}}\sqrt{\frac{C_2}{\varepsilon_{k_s}}}. \label{lim2}
\end{equation}

Taking the limit of both sides of (\ref{lim2}) gives $||\bar{x}||\leq ||x^*||$. Thus, by the uniqueness of the least-norm solution $x^*$ and the fact that $\bar{x}$ is also a solution to VI($Q$, $F$\textsc{}), any sequence $x_{\delta\varepsilon}^{k}$ has exactly one accumulation point. According to the boundedness of $x_{\delta\varepsilon}^{k}$, $\lim_{k\rightarrow\infty}x_{\delta\varepsilon}^{k}=x^*$. Since $x_{\delta\varepsilon}^{k}$ is arbitrary,  $\lim_{t\rightarrow\infty}x_t^*=x^*$. The proof is completed.

According to (\ref{lim1}), we can further get the following conclusion: for any $t\geq 0$,
\begin{equation} 
||x_t^*||^2-||x^*||^2\leq \frac{C_3}{\delta(t)}\sqrt{\frac{C_2}{\varepsilon(t)}}, \label{lim12}
\end{equation} 
which will be used in the following convergence rate analysis.
\end{proof}

%
%
%
%
%
%
%

\begin{remark}
   The idea of the proof of Lemma \ref{relationship} is motivated from Theorem 2 of \cite{clempner2018tikhonov}. However, \cite{clempner2018tikhonov} solves a polylinear programming problem which is inconsistent with the studied problem in this work. Theorem 5.1 of  \cite{konnov2015regularized} provides a similar conclusion as Lemma \ref{relationship}. However, the assumptions in this work don't satisfy the conditions (C2 or C2') in  Theorem 5.1 of \cite{konnov2015regularized} and the proof methodology is different. 
	
\end{remark}

\color{black}
\begin{remark}
 If  $\lim_{t\rightarrow\infty}(x(t)-x_t^*)=0$, then according to Lemma \ref{relationship}, $\lim_{t\rightarrow\infty}x(t)=x^*$. Thus, the problem is transformed to design an algorithm such that $\lim_{t\rightarrow\infty}(x(t)-x_t^*)=0$, i.e., to design an algorithm to track a solution $x_t^*$ of  VI($\Omega, \Phi_{\delta\varepsilon}$). VI($\Omega, \Phi_{\delta\varepsilon}$) is a regularized penalized VI with time-varying coefficients. The algorithms to deal with a regularized VI without penalty terms have been proposed in \cite{facchinei2007finite}\cite{koshal2010single} using discrete-time algorithms and \cite{cominetti2008strong} using  continuous-time differential equations. However, when considering constraints and penalty functions, the analysis will be different. Here, before designing the algorithm, we first generalize an existing result (appeared in Lemma 3 of \cite{koshal2010single}, Lemma 2.3 of \cite{ boct2020inducing}, and Page 529 of \cite{attouch1996dynamical}, etc) to the penalized case, which plays an important role in the convergence analysis of a regularized VI algorithm.
	\end{remark}

	\begin{lemma} Under the conditions in Lemma \ref{relationship}, we have (a) $\dot{x}_t^* $ exists almost everywhere; (b) there exist positive constants $c_x,c_p$ such that
		$||\dot{x}_t^*||\leq \frac{c_x|\dot{\delta}(t)|+c_p|\dot{\varepsilon}(t)|}{\delta(t)}$ holds almost everywhere. \label{xbound}
	\end{lemma}

	\begin{proof}	
		(a) It suffices to prove that ${x}_t^*$ is locally Lipschitz continuous with respect to both $\delta$ and $\varepsilon$. Let $\delta_1>0$, $\delta_2>0$, and $x_{\delta_1}^*$ and $x_{\delta_2}^*$ be the solutions to VI($\Omega, \Phi_{\delta_1\varepsilon}$) and  VI($\Omega, \Phi_{\delta_2\varepsilon}$), respectively. According to (\ref{sol}),
		 \begin{equation}
		(x_{\delta_2}^*-x_{\delta_1}^*)^T(F(x_{\delta_1}^*)+\delta_1x_{\delta_1}^*+\varepsilon\nabla_x P(x_{\delta_1}^*))\geq 0, \label{e11}
		\end{equation}
		and 
		\begin{equation}
		(x_{\delta_1}^*-x_{\delta_2}^*)^T(F(x_{\delta_2}^*)+\delta_2x_{\delta_2}^*+\varepsilon\nabla_x P(x_{\delta_2}^*))\geq 0. \label{e21}
		\end{equation}
		
		According to the monotonicity of $F$ and $\nabla_x P$,  adding (\ref{e11}) to (\ref{e21}) gives 
		\begin{equation}
		(x_{\delta_2}^*-x_{\delta_1}^*)^T(\delta_1x_{\delta_1}^*-\delta_2x_{\delta_2}^*)\geq 0,
			\end{equation}
			which implies that $||x_{\delta_2}^*-x_{\delta_1}^*||\leq \frac{||x_{\delta_1}^*||}{\delta_2}|\delta_2-\delta_1|$. Thus,  ${x}_t^*$ is locally Lipschitz continuous in $\delta$.
			
			Similarly,  let $\varepsilon_1>0$,  $\varepsilon_2>0$, and $x_{\varepsilon_1}^*$ and $x_{\varepsilon_2}^*$ be the solutions to VI($\Omega, \Phi_{\delta\varepsilon_1}$) and  VI($\Omega, \Phi_{\delta \varepsilon_2}$), respectively. According to (\ref{sol}),
			 \begin{equation}
			(x_{\varepsilon_2}^*-x_{\varepsilon_1}^*)^T(F(x_{\varepsilon_1}^*)+\delta x_{\varepsilon_1}^*+\varepsilon_1\nabla_x P(x_{\varepsilon_1}^*))\geq 0, \label{e31}
			\end{equation}
			and 
			\begin{equation}
			(x_{\varepsilon_1}^*-x_{\varepsilon_2}^*)^T(F(x_{\varepsilon_2}^*)+\delta x_{\varepsilon_2}^*+\varepsilon_2\nabla_x P(x_{\varepsilon_2}^*))\geq 0. \label{e32}
			\end{equation}
			
			According to the monotonicity of $F$ and $\nabla_x P$,  adding (\ref{e31}) to (\ref{e32}) gives 
			 \begin{eqnarray}
			(x_{\varepsilon_2}^*-x_{\varepsilon_1}^*)^T(\varepsilon_1-\varepsilon_2 )\nabla_x P(x_{\varepsilon_1}^*)\geq \delta ||x_{\varepsilon_2}^*-x_{\varepsilon_1}^*||^2,
			\end{eqnarray}
			which implies that $||x_{\varepsilon_2}^*-x_{\varepsilon_1}^*||\leq \frac{||\nabla_x P(x_{\varepsilon_1}^*)||}{\delta}|\varepsilon_1-\varepsilon_2 |$.  Thus,  ${x}_t^*$ is locally Lipschitz continuous in $\varepsilon$.
			\color{black}
			
	 (b) Let $t_1=t_2+h$. Then, according to (\ref{sol}),
	 \begin{eqnarray}
	 (x_{t_2}^*-x_{t_1}^*)^T(F(x_{t_1}^*)+\delta(t_1)x_{t_1}^*+\varepsilon(t_1)\nabla_x P(x_{t_1}^*))\geq 0, \label{e1}
	 \end{eqnarray}
	 and 
	 	 \begin{eqnarray}
	 (x_{t_1}^*-x_{t_2}^*)^T(F(x_{t_2}^*)+\delta(t_2)x_{t_2}^*+\varepsilon(t_2)\nabla_x P(x_{t_2}^*))\geq 0, \label{e2}
	 \end{eqnarray}
	 with $x_{t_1}^*$ and $x_{t_2}^*$ being the solutions to VI($\Omega, \Phi_{\delta\varepsilon}$) at $t_1$ and $t_2$, respectively.
	 
	 According to the monotonicity of $F$,  adding (\ref{e1}) to (\ref{e2}) gives 
	  \begin{align}
	 (x_{t_2}^*-x_{t_1}^*)^T(\delta(t_1)x_{t_1}^*+\varepsilon(t_1)\nabla_x P(x_{t_1}^*)\notag\\-(\delta(t_2)x_{t_2}^*+\varepsilon(t_2)\nabla_x P(x_{t_2}^*)))
	 \geq 0. \label{e3}
	 \end{align}
	 
	 Then, 
	 \begin{align}
	 (x_{t_2}^*-x_{t_1}^*)^T(\delta(t_2)(x_{t_1}^*-x_{t_2}^*)+(\delta(t_1)-\delta(t_2))x_{t_1}^*\notag\\
	 +\varepsilon(t_1)(\nabla_x P(x_{t_1}^*)-\nabla_x P(x_{t_2}^*))\notag \\
	 +(\varepsilon(t_1)-\varepsilon(t_2))\nabla_x P(x_{t_2}^*)))
	 \geq 0, \label{e4}
	 \end{align}
	 which implies that $(x_{t_1}^*-x_{t_2}^*)^T\delta(t_2)(x_{t_1}^*-x_{t_2}^*)\leq (x_{t_2}^*-x_{t_1}^*)^T(\delta(t_1)-\delta(t_2))x_{t_1}^*+(x_{t_2}^*-x_{t_1}^*)^T(\varepsilon(t_1)-\varepsilon(t_2))\nabla_x P(x_{t_2}^*)\leq c_x||x_{t_2}^*-x_{t_1}^*|||\delta(t_1)-\delta(t_2)|+c_p||x_{t_2}^*-x_{t_1}^*|||\varepsilon(t_1)-\varepsilon(t_2)|$, where $c_x$ and $c_p$ are  positive constants satisfying $||x||\leq c_x$ and $||\nabla_x P(x)||\leq c_p$ for all $x\in \Omega$.
	 
	 Thus, $\frac{||x_{t_2}^*-x_{t_1}^*||}{h}\leq\frac{c_x|\delta(t_1)-\delta(t_2)|+c_p|\varepsilon(t_1)-\varepsilon(t_2)|}{\delta(t_2)h}$. Letting $h \rightarrow 0$ and taking
	limits on the inequality gives the conclusion.
	\end{proof}

	\subsection{Full-Decision Information Algorithm: Projected Gradient based Regularized Penalized Dynamics}
	
	 There have been some continuous-time dynamical systems that were proposed to solve an optimization problem (e.g.,\cite{ brown1989some}). However, as remarked in Page 1110 of  \cite{facchinei2007finite}, despite that it is from a discrete-time perspective, when extending to the noncooperative game/VI problem, these algorithms may fail to converge. For example,  the projected gradient decent method  converges for the convex optimization problem but is not necessarily convergent for the monotone VI problem. First, we review several centralized dynamical systems that were proposed  for the noncooperative game/VI  problem.
	
	In \cite{bruck1975asymptotic}, the differential inclusion
	\begin{equation}
	\dot{x}\in-T(x)  \label{dynamics001}
	\end{equation}
	for $T(x)$ being a demipositive  map (which is a more restrictive assumption than maximal monotone) was investigated.
	
	In \cite{cominetti2008strong},  the dynamical system 
	\begin{equation}
	\dot{x}\in-T(x)-\delta(t)x  \label{dynamics002}
	\end{equation}
	for $T(x)$ being a maximal monotone map was investigated, and asymptotic strong convergence was proven.
	
	When there exist constraints in the noncooperative game/VI  problem,  the authors in \cite{attouch2010asymptotic}  proved ergodic convergence for the following dynamics:
		\begin{equation}
	\dot{x}\in-T(x)-\beta(t)\partial \Phi(x), \label{dynamics003}
	\end{equation}
	where $T(x)$ is  maximal monotone, $\beta(t)$ is time-varying and $\Phi(x)$ is a convex function.
	
	In \cite{abbas2015dynamical}, the authors proved the asymptotic convergence of the following proximal-gradient dynamics 
			\begin{equation}
	\dot{x}=-x+\text{prox}_{\mu\Phi}(x-\mu T(x)), \label{dynamics004}
	\end{equation}
	where $\mu>0 $, $T(x)$ is a cocoercive (similar to strongly monotone) operator, and ``$\text{prox}$'' is the proximal point operator of a convex function $\Phi$.
	
	In a recent work \cite{boct2020inducing}, several types of Tikhonov regularized dynamics were proposed including the following one: 
		\begin{align}
	\dot{x}=&-x+z+\gamma(t)(T(x)-T(z)+\varepsilon(t)(x-z)), \notag \\
	z=&J_{\gamma(t)\mathcal{A}}[x-\gamma(t)(T(x)+\varepsilon(t)x)],
	 \label{dynamics005}
	\end{align}
	where $\gamma(t)$ and $\varepsilon(t)$ are time-varying parameters, $T(x)$ is a monotone and Lipschitz map, $\mathcal{A}(x)$ is a maximal monotone map, and $J_{\gamma(t)\mathcal{A}}=(I+\gamma(t)\mathcal{A})^{-1}$ is the resolvent operator. 
	
	For more works, we refer the readers to the above literature and the references therein.
	
	Among the above dynamical systems, (\ref{dynamics005}) has the potential to be applied to solve GNEP (\ref{P1}). However, this method may be not easy to be adopted in a distributed problem. Moreover, a simpler dynamical system may be expected.
	
	In this section, we propose a regularized and penalized dynamical equation and prove the asymptotic convergence, which can be described as follows:
	\begin{eqnarray}
	\dot{x}_i&=&\varsigma(t)(\mathcal{P}_{\Omega_i}[x_i-\gamma(t)(\nabla_{x_{i}} f_{i}(x)+\varepsilon(t) \nabla _{x_i}P(x) \notag \\& &+ \delta(t)x_i)]-x_i), \label{dynamics01}
	\end{eqnarray}
	which is a full-decision information algorithm, where $\delta(t)$, $\varepsilon(t)$, $\varsigma(t)$ and $\gamma(t)$ are time-varying, positive and smooth parameters with bounded initial values, and $P(x)$ is a penalty function defined in (\ref{penalty}). 
	
	The concatenated form of (\ref{dynamics01}) can be written as
		\begin{eqnarray}
	\dot{x} &=&\varsigma(t)(\mathcal{P}_{\Omega }[x -\gamma(t)\Phi_{\delta\varepsilon}(x)]-x),
	\end{eqnarray}
	where $\Phi_{\delta\varepsilon}(x)$ was defined in (\ref{phiy}).

Then, we arrive at the main conclusion of this session as follows:
	
\begin{theorem}
	Under Assumptions \ref{fassumption}--\ref{gassumption2}, the algorithm given in (\ref{dynamics01})  asymptotically converges to the least-norm variational equilibrium of GNEP (\ref{P1}), provided that the following conditions hold: 
		\begin{eqnarray}
			0<r_1(t)<1, \lim_{t\rightarrow\infty} \frac{r_2(t)}{{r}_1(t)\varsigma(t)} =0, 
\notag \\
\lim_{t\rightarrow\infty}\int_{s=0}^tr_1(s)\varsigma(s)ds=\infty, \notag \\
	\int_{s=0}^{t} r_2(s)e^{E_{r_1\varsigma}(s)}ds \text{ is bounded or} \notag \\
\lim_{t\rightarrow\infty}\int_{s=0}^{t} r_2(s)e^{E_{r_1\varsigma}(s)}ds=\infty; 	\label{b1} \\
	\frac{\int_{s=0}^{t} r_2(s)e^{E_{r_1\varsigma}(s)}ds}{e^{E_{r_1\varsigma}(t)}}\leq c_0; \label{b2}  \\
    -\gamma^2(t)(Nb_1^2+Nb_2^2\varepsilon^2(t)+\delta^2(t))\notag \\
    +2\gamma(t)\delta(t)=r_1(t) , \notag \\
    \frac{ |\dot{\delta}(t)|+ |\dot{\varepsilon}(t)|}{\delta(t)}=r_2(t), 
    \lim_{t\rightarrow\infty} \delta(t)=0,  \notag \\
    \lim_{t\rightarrow\infty} \delta^2(t)\varepsilon(t)=\infty,  \lim_{t\rightarrow\infty} \varepsilon(t)=\infty,  \label{C0}
		\end{eqnarray}
	 where $E_{r_1\varsigma}(t)=\int_{s=0}^t \frac{r_1(s)\varsigma(s)}{2}ds$, $b_1$ is a local Lipschitz constant dependent on the initial values, $b_2$ is a global Lipschitz constant, and $c_0$ is a positive constant independent of $b_1$. \label{theorem1}
\end{theorem}
\begin{proof}
Let $q(t)=x-x_t^*$. Define a Lyapunov candidate function $V(t,q)=\frac{1}{2}q^Tq$. 

According to Lemma  \ref{projector},
\begin{eqnarray}
& &||\mathcal{P}_{\Omega}[x-\gamma(t)\Phi_{\delta\varepsilon}(x)]-x_t^*||^2 \notag\\
&=&  ||\mathcal{P}_{\Omega}[x-\gamma(t)\Phi_{\delta\varepsilon}(x)]-\mathcal{P}_{\Omega}[x_t^*-\gamma(t)\Phi_{\delta\varepsilon}(x_t^*)]||^2\notag\\
&\leq&||x-x_t^*-\gamma(t) (\Phi_{\delta\varepsilon}(x)-\Phi_{\delta\varepsilon}(x_t^*))||^2 \notag\\
&=&||x-x_t^*||^2-2\gamma(t)(x-x_t^*)^T(\Phi_{\delta\varepsilon}(x)-\Phi_{\delta\varepsilon}(x_t^*))\notag\\
& &+\gamma^2(t)||\Phi_{\delta\varepsilon}(x)-\Phi_{\delta\varepsilon}(x_t^*)||^2.  \label{d1} 
\end{eqnarray}

Let $\Upsilon_1=\{q\in\mathbb{R}^{N }|||q||
\leq ||q(0)||+c\}
$, where $c$ is a positive constant satisfying $c_{px}c_0\leq c$ and $c_{px}=\min\{c_{p}, c_x\}$. According to the boundedness of $x_t^*$, and the assumption that $f_i$ is twicely continuously differentiable, there exists a positive constant $b_1$ such that for any $x\in \Upsilon_{11}\triangleq\{x\in\mathbb{R}^N|||x||\leq ||q(0)||+c+c_x\} $,
\begin{eqnarray}
| \nabla_{x_{i}} f_{i}(x)-\nabla_{x_{i}} f_{i}(x_t^*)|\leq b_1||x-x_t^*||. \label{fLipschitz01}
\end{eqnarray}

Furthermore, it can be verified that $\nabla_{x_{i}} P(x)$ is globally Lipschitz continuous since the constraints are linear (by using compose function  properties). Thus, there exists a global Lipschitz constant  $b_2>0$, such that
\begin{eqnarray}
| \nabla_{x_{i}} P(x)-\nabla_{x_{i}} P(x_t^*)|\leq b_2||x-x_t^*||, \label{fLipschitz012}
\end{eqnarray}
which implies that 
\begin{eqnarray}
& &||\Phi_{\delta\varepsilon}(x)-\Phi_{\delta\varepsilon}(x_t^*)||^2 \notag \\
& \leq& (Nb_1^2+Nb_2^2\varepsilon^2(t)+\delta^2(t))||x-x_t^*||^2 .
\end{eqnarray}

Moreover, it can be obtained that if $q(t)\in \Upsilon_{1}$,  $x(t)\in \Upsilon_{11}$. Thus, based on (\ref{d1}), for all $q(t)\in \Upsilon_{1}$,
	\begin{eqnarray}
||\mathcal{P}_{\Omega}[x-\gamma(t)\Phi_{\delta\varepsilon}(x)]-x_t^*||^2 \notag\\
	\leq  (1-r_1(t))||x-x_t^*||^2,
\end{eqnarray}
where $r_1(t)=-\gamma^2(t)(Nb_1^2+Nb_2^2\varepsilon^2(t)+\delta^2(t))+2\gamma(t)\delta(t)$ as defined in the theorem. \footnote{To make it easier for the reader to follow this work, we define  $r_1(t)$ again in the proof although it was already defined in the theorem.}

Since $q(0)\in\Upsilon_1$, we prove that  $q(t)\in\Upsilon_{1}$ for all $t\geq 0$ by proof of contradiction. Suppose that there exists a time instant $T$ such that $q(t)\in\Upsilon_1$ for all $0\leq t<T$ and $q(T)\notin\Upsilon_1$ .

Taking the derivative of $V$ gives 
	\begin{eqnarray}
\dot{V}&\overset{a.e.}=&(x-x_t^*)^T(\dot{x}-\dot{x}_t^*)\notag \\
&\overset{a.e.}=&(x-x_t^*)^T(\varsigma(t)(\mathcal{P}_{\Omega}[x-\gamma(t)\Phi_{\delta\varepsilon}(x)]-x_t^*)\notag \\
& &-\varsigma(t)(x-x_t^*))-(x-x_t^*)^T\dot{x}_t^*\notag\\
&\overset{a.e.}\leq&-\frac{1}{2}r_1(t)\varsigma(t)||x-x_t^*||^2-(x-x_t^*)^T\dot{x}_t^* \notag \\
&\overset{a.e.}\leq&-r_1(t)\varsigma(t)V+\sqrt{2V}||\dot{x}_t^*||, 
\end{eqnarray}
for all $0 \leq t< T$.

Define  $\phi=\sqrt{2V}=||q||$. Taking the derivative of $\phi$ gives 
	\begin{eqnarray}
\dot{\phi}&\overset{a.e.}\leq& -\frac{r_1(t)\varsigma(t)}{2}\phi +||\dot{x}_t^*|| \notag\\
&\overset{a.e.}\leq& -\frac{r_1(t)\varsigma(t)}{2}\phi+c_{px} r_2(t), \label{x1}
\end{eqnarray}
for all $0\leq t<T$ by using Lemma \ref{xbound}, where $r_2(t)=\frac{ |\dot{\delta}(t)|+ |\dot{\varepsilon}(t)|}{\delta(t)}$.

Define an energy function $E_{r_1\varsigma}(t)\triangleq \int_{s=0}^t \frac{r_1(s)\varsigma(s)}{2}ds$. Then, 
\begin{equation}
\frac{d}{dt}({\phi}e^{E_{r_1\varsigma}(t)})\overset{a.e.}\leq c_{px}r_2(t) e^{E_{r_1\varsigma}(t)}, \label{integrated1}
\end{equation}
for all $0\leq t<T$.

Integrating both sides of (\ref{integrated1}) from 0 to $T$ gives
\begin{equation}
\phi(T)e^{E_{r_1\varsigma}(T)}-\phi(0)\leq \int_{s=0}^{T}c_{px}r_2(s)e^{E_{r_1\varsigma}(s)}ds,
\end{equation}
which implies
\begin{eqnarray}
0\leq\phi(T)\leq e^{-E_{r_1\varsigma}(T)}(\phi(0)+c_{px}\int_{s=0}^{T} r_2(s)e^{E_{r_1\varsigma}(s)}ds). \label{lip1}
\end{eqnarray}

According to (\ref{lip1}), $q(T)\in\Upsilon_2\triangleq\{q\in\mathbb{R}^{N}|||q||
\leq  ||q(0)|| +c_{px}c_0 \}\subseteq\Upsilon_1$. This contradicts with the assumption $q(T)\notin\Upsilon_1$.

Thus, it can be obtained that $q(t)\in\Upsilon_1$ for all $t\geq0$. Then, for all $t\geq 0$,
\begin{eqnarray}
0\leq\phi(t)\leq e^{-E_{r_1\varsigma}(t)}(\phi(0)+c_{px}\int_{s=0}^{t} r_2(s)e^{E_{r_1\varsigma}(s)}ds). \label{lip12}
\end{eqnarray}

Since $\lim_{t\rightarrow\infty}\int_{s=0}^tr_1(s)\varsigma(s)ds=\infty$, $\lim_{t\rightarrow\infty}e^{-E_{r_1\varsigma}(t)}=0$. If $\int_{s=0}^{t} r_2(s)e^{E_{r_1\varsigma}(s)}ds$ is bounded, $\lim_{t\rightarrow\infty}\frac{\int_{s=0}^{t} r_2(s)e^{E_{r_1\varsigma}(s)}ds}{e^{E_{r_1\varsigma}(t)}}=0$. If $\lim_{t\rightarrow\infty}\int_{s=0}^{t} r_2(s)e^{E_{r_1\varsigma}(s)}ds$ equals to $\infty$, according to L'Hospital's rule, 
\begin{align}
\lim_{t\rightarrow\infty}\frac{\int_{s=0}^{t} r_2(s)e^{E_{r_1\varsigma}(s)}ds}{e^{E_{r_1\varsigma}(t)}}\notag\\
=\lim_{t\rightarrow\infty}\frac{r_2(t)e^{E_{r_1\varsigma}(t)} }{\frac{e^{E_{r_1\varsigma}(t)} {r}_1(t)\varsigma(t)}{2}}=0. \label{lim0}
\end{align}

Then, $\lim_{t\rightarrow\infty}\frac{\int_{s=0}^{t} r_2(s)e^{E_{r_1\varsigma}(s)}ds}{e^{E_{r_1\varsigma}(t)}}=0$.  According to  (\ref{lip12}), 	$\lim_{t\rightarrow\infty}\phi(t)=0$.

Thus, $\lim_{t\rightarrow\infty}||x-x_t^*||=0$. Based on Lemma \ref{relationship}, $\lim_{t\rightarrow\infty}||x-x^*||=0$.

	\end{proof}
\begin{remark}
	The solutions to the conditions in Theorem \ref{theorem1} exist. Let  $\varsigma(t)=1$, $r_1(t)=\frac{(1+t)^{-\frac{1}{7}}}{N+N(1+t)^{ \frac{1}{3}}+(1+t)^{-\frac{1}{7}}}$, $r_2(t)=\frac{\frac{1}{14}(1+t)^{-\frac{15}{14}}+\frac{1}{6b_2}(1+t)^{-\frac{5}{6}}}{(1+t)^{-\frac{1}{14}}}$, which satisfy the conditions in (\ref{b1}) and (\ref{b2}). Then,  $\delta(t)=b_1(1+t)^{-\frac{1}{14}}$,  $\varepsilon(t)=\frac{b_1}{b_2}(1+t)^{\frac{1}{6}}$, $\gamma(t)=\frac{\delta (t)}{Nb_1^2+Nb_2^2\varepsilon^2(t)+\delta^2(t)}$ satisfy the conditions in (\ref{C0}). \label{remark1} 
  \end{remark}

\begin{remark}
The introduction  of  $\varsigma(t)$ enables more abundant selection of parameters. For example, (a) $\varsigma(t)=(1+t)^6$, $\delta(t)=b_1(1+t)^{-\frac{1}{2}}$,  $\varepsilon(t)=\frac{b_1}{b_2}(1+t)^{1.2}$, $\gamma(t)=\frac{\delta (t)}{Nb_1^2+Nb_2^2\varepsilon^2(t)+\delta^2(t)}$; (b) $\varsigma(t)=e^{4(a+b)t}$, $\delta(t)=b_1e^{-at}$,  $\varepsilon(t)=\frac{b_1}{b_2}e^{bt}$, $\gamma(t)=\frac{\delta (t)}{Nb_1^2+Nb_2^2\varepsilon^2(t)+\delta^2(t)}$  with $b>2a>0$.  \label{r5}
\end{remark}

\begin{remark}
The condition in (\ref{b2}) is mainly used to enable an upper bound independent of the local Lipschitz constant $b_1$, while the boundedness of $\frac{\int_{s=0}^{t} r_2(s)e^{E_{r_1\varsigma}(s)}ds}{e^{E_{r_1\varsigma}(t)}}$ can be  automatically  satisfied provided that  the conditions  in (\ref{b1}) and (\ref{C0}) hold. The condition in (\ref{b2})  may restrict the selection of the parameters. For the examples given in Remark \ref{r5}, three parameters $\delta(t)$, $\varepsilon(t)$, and $\gamma(t)$ are related to the local Lipschitz constant $b_1$. If the gradients are globally Lipschitz continuous instead of locally Lipschitz continuous, this condition can be removed, as shown in the following lemma.
\end{remark}

\begin{proposition}
Under Assumptions \ref{fassumption}--\ref{gassumption2} and assuming that $\nabla_{x_i}f_i(x)$ is globally Lipschitz continuous, the algorithm given in (\ref{dynamics01})  asymptotically converges to the least-norm variational equilibrium of GNEP (\ref{P1}), provided that the conditions in  (\ref{b1}) and (\ref{C0}) hold. \label{theorem2}
 \end{proposition}
\begin{proof}
	The proof is similar to the proof of Theorem \ref{theorem1}. Note that the uniform boundedness of $\frac{\int_{s=0}^{t} r_2(s)e^{E_{r_1\varsigma}(s)}ds}{e^{E_{r_1\varsigma}(t)}}$ can be guaranteed by the fact that $\lim_{t\rightarrow\infty}\frac{\int_{s=0}^{t} r_2(s)e^{E_{r_1\varsigma}(s)}ds}{e^{E_{r_1\varsigma}(t)}}=0$ under the condition in (\ref{b1}).
\end{proof}
\begin{remark}
	 It can be easily verified that the parameters  $\varsigma(t)=(1+t)^6$, $\delta(t)= (1+t)^{-\frac{1}{2}}$,  $\varepsilon(t)= (1+t)^{1.2}$, $\gamma(t)=\frac{\delta (t)}{Nb_1^2+Nb_2^2\varepsilon^2(t)+\delta^2(t)}$ satisfy the conditions.  In this case, only $\gamma(t)$ is dependent on $b_1$ and $b_2$.
	\end{remark}

\begin{remark}
	The local Lipschitz constants $b_1$ may be unknown. One can choose a sufficiently large parameter according to (\ref{fLipschitz01}). In addition, in many scenarios, the Lipschitz constants may be estimated according to the problem data. The dependency on the Lipschitz constants is common in the existing literature. It is noted that recently there are some works investigating adaptive methods to remove this dependency. However, the extension of this technique to the proposed method is not easy. This is partially caused by the time-varying property of the parameters. The projection operator may also affect the usage of the adaptive technique.
\end{remark}

\color{black}
 
	\subsection{Partial-Decision Information Algorithm }

In this section, we consider that each player  can only have access to the state information of its neighbors.

The following distributed updating law is designed:
\begin{align}
\dot{x}_i=&\varsigma(t)(\mathcal{P}_{\Omega_i}[x_i-\gamma(t)(\nabla_{x_{i}} f_{i}(y_i)+\varepsilon(t) \nabla _{x_i}P(y_i) \notag \\
 &+ \delta(t)x_i)]-x_i), \notag \\
\dot{y}_{ij}= &-w(t)\sum_{k=1}^{N}a_{ik}(y_{ij}-y_{kj}),j\in\mathcal{V}\backslash i, \label{dynamics1}
\end{align}
where  $y_{ij}\in\mathbb{R}$,  $y_{ii}=x_i$, $y_i=[y_{ij}]_{j\in\mathcal{V}}$,  $\varsigma(t)$, $\delta(t)$, $\varepsilon(t)$, $\gamma(t)$ and $w(t)$ are time-varying, positive and smooth parameters with bounded initial values, and $P(\cdot)$ is a penalty function defined in  (\ref{penalty}).
\color{black}

For agent $i$, define a subgraph $\mathcal{G}_{i}=\left\{ \mathcal{V}_{i},\mathcal{E}_{i}\right\} $,
where $\mathcal{V}_{i}=\mathcal{V}\backslash\{i\}$ and $\mathcal{E}_{i}\subset\mathcal{V}_{i}\times\mathcal{V}_{i}$
denote the set of vertices and edges, respectively. Let $L_{i}$ be
the Laplacian matrix of $\mathcal{G}_{i}$ and $B_{i}=\text{{diag}}\{a_{1i},\cdots,a_{(i-1)i},a_{(i+1)i},\cdots,a_{Ni}\}\in\mathbb{R}^{(N-1)\times(N-1)}$. According to Assumption \ref{graph} and Theorem 4 of \cite{Khoo}, it can be obtained that the matrix $L_{i}+B_{i}$ is symmetric and positive definite.

Let $Y=[y_i]_{i\in\mathcal{V}}\in\mathbb{R}^{N^2}$, $\bar{Y}_{i}=[y_{ji}]_{j\in\mathcal{V}\backslash i}\in\mathbb{R}^{N-1}$ and $\hat{Y}_i=\bar{Y}_{i}-x_i\mathbf{1}_{N-1}\in\mathbb{R}^{N-1}$. The concatenated form of (\ref{dynamics1}) can be written as
	\begin{align}
\dot{x} =&\varsigma(t)(\mathcal{P}_{\Omega }[x -\gamma(t)\Phi_{\delta\varepsilon}(Y)]-x),\notag \\
\dot{\bar{Y}}_{i}  =&-w(t)(L_{i}+B_{i})\hat{Y}_{i},i\in\mathcal{V},
\end{align}
where $\Phi_{\delta\varepsilon}(Y)=F(Y)+\delta(t)x+\varepsilon(t)\nabla_x P(Y)$, $F(Y)\triangleq \left[ \nabla_{x_{i}} f_{i}(y_i)\right]_{i\in\mathcal{V}}$ and $\nabla_x P(Y)\triangleq \left[ \nabla_{x_{i}} P (y_i)\right]_{i\in\mathcal{V}}$.

\begin{remark}
According to (\ref{dynamics1}), the algorithm for player $i$ uses the local states $x_i$ and $y_{ij}$, and the neighboring states $y_{kj}$ for $k\in\mathcal{V}_i$. Thus, the algorithm uses only partial-decision information. 
\end{remark}

%

The main result of this section can be described as follows:

\color{black}
\begin{theorem}
	Under Assumptions \ref{fassumption}--\ref{graph}, the algorithm given in (\ref{dynamics1})  asymptotically converges to the least-norm variational equilibrium of GNEP (\ref{P1}), provided that the following conditions hold
	\begin{eqnarray}
	0<u_1(t)<1, \lim_{t\rightarrow\infty} \frac{r_2(t)}{{u}_1(t)\varsigma(t)}dt=0, 
	\notag \\
	\lim_{t\rightarrow\infty}\int_{s=0}^tu_1(s)\varsigma(s)ds=\infty, \notag \\
	\int_{s=0}^{t} r_2(s)e^{E_{u_1\varsigma}(s)}ds \text{ is bounded or} \notag \\
	\lim_{t\rightarrow\infty}\int_{s=0}^{t} r_2(s)e^{E_{u_1\varsigma}(s)}ds=\infty; 	\label{b12} \\
	\frac{\int_{s=0}^{t} r_2(s)e^{E_{u_1\varsigma}(s)}ds}{e^{E_{u_1\varsigma}(t)}}\leq c_0'; \label{b22}  \\
	-\gamma^2(t)(Nb_1'^2+Nb_2^2\varepsilon^2(t)+\delta^2(t)+1)\notag \\
	+2\gamma(t)\delta(t)=u_1(t) , \notag \\
	\frac{ |\dot{\delta}(t)|+ |\dot{\varepsilon}(t)|}{\delta(t)}=r_2(t), 
	\lim_{t\rightarrow\infty} \delta(t)=0,  \notag \\
	\lim_{t\rightarrow\infty} \delta^2(t)\varepsilon(t)=\infty,  \lim_{t\rightarrow\infty} \varepsilon(t)=\infty;  \label{C1} \\
	\theta(t)\geq \frac{1}{4}u_1(t)\varsigma(t), \label{C2}
	\end{eqnarray}
	where $E_{u_1\varsigma}(t)=\int_{s=0}^t \frac{u_1(s)\varsigma(s)}{4}ds$,   $\theta(t)=w(t)\lambda_{\text{min}}-\frac{u_2(t)\varsigma(t)}{2}-\frac{\varsigma (t)(2-u_1(t))(N-1)^2}{u_1(t)}-\frac{u_1(t)u_2(t)\varsigma(t)}{4(2-u_1(t))},$  $u_2(t)= (\gamma^2(t)+1)(b_3^2+b_2^2\varepsilon^2(t))$, $b_1'$ and $b_3$ are two local Lipschitz constants dependent on the initial values, $b_2$ is a global Lipschitz constant,  $c_0'$ is a positive constant independent of $b_1'$ and $b_3$, and $\lambda_{\min}$ is the minimum value among the eigenvalues of $L_i+B_i$, $i\in\mathcal{V}$. \label{t1}
\end{theorem}
\color{black}
\begin{proof}
	\color{black}
	Let $z=[(x-x_t^*)^T,\hat{Y}_1^T,\cdots,\hat{Y}_N^T]^T$. For the nonautonomous system (\ref{dynamics1}), define a Lyapunov candidate function $V(t,z)=\frac{1}{2}z^Tz=\frac{1}{2}(x-x_t^*)^T(x-x_t^*)+\frac{1}{2}\sum_{i=1}^N\hat{Y}_i^T\hat{Y}_i.$ 
	
	According to Lemma  \ref{projector},
	\begin{eqnarray}
	& &||\mathcal{P}_{\Omega}[x-\gamma(t)\Phi_{\delta\varepsilon}(Y)]-x_t^*||^2 \notag\\
	&=&  ||\mathcal{P}_{\Omega}[x-\gamma(t)\Phi_{\delta\varepsilon}(Y)]-\mathcal{P}_{\Omega}[x_t^*-\gamma(t)\Phi_{\delta\varepsilon}(x_t^*)]||^2\notag\\
	&\leq&||x-x_t^*-\gamma(t) (\Phi_{\delta\varepsilon}(Y)-\Phi_{\delta\varepsilon}(x_t^*))||^2 \notag\\
	&=&||x-x_t^*||^2-2\gamma(t)(x-x_t^*)^T(\Phi_{\delta\varepsilon}(Y)-\Phi_{\delta\varepsilon}(x_t^*))\notag\\
		& &+\gamma^2(t)||\Phi_{\delta\varepsilon}(Y)-\Phi_{\delta\varepsilon}(x_t^*)||^2 \notag \\
		&=&||x-x_t^*||^2-2\gamma(t)(x-x_t^*)^T(\Phi_{\delta\varepsilon}(x)-\Phi_{\delta\varepsilon}(x_t^*))\notag\\
		& &-2\gamma(t)(x-x_t^*)^T(\Phi_{\delta\varepsilon}(Y)-\Phi_{\delta\varepsilon}(x)) \notag \\
			& &+\gamma^2(t)||\Phi_{\delta\varepsilon}(Y)-\Phi_{\delta\varepsilon}(x_t^*)||^2. \label{r12}
	\end{eqnarray}

	Let $\Delta_1=\{z\in\mathbb{R}^{N^2}|||z||
	\leq ||z(0)|| +c'\}
	$, where $c'$ is a positive constant satisfying $c_{px}c_0'\leq c'$ and $c_0'$ was defined in (\ref{b22}).	Then, there exist positive constants $b_1'$ and $b_2 $ such that for all $x\in \Delta_{11}\triangleq \{x\in\mathbb{R}^N|||x||\leq ||z(0)|| +c' +c_x \}$,
		\begin{eqnarray}
		| \nabla_{x_{i}} f_{i}(x)-\nabla_{x_{i}} f_{i}(x_t^*)|\leq b_1'||x-x_t^*||, \label{fLipschitz1}
		\end{eqnarray}
	and for all $x\in\mathbb{R}^N$,
				\begin{eqnarray}
		| \nabla_{x_{i}} P(x)-\nabla_{x_{i}} P(x_t^*)|\leq b_2||x-x_t^*||, \label{fLipschitz12}
		\end{eqnarray}
		which implies that for  all $z\in\Delta_1$,
			\begin{eqnarray}
		& &||\Phi_{\delta\varepsilon}(x)-\Phi_{\delta\varepsilon}(x_t^*)||^2 \notag \\
		&\leq& (Nb_1'^2+Nb_2^2\varepsilon^2(t)+\delta^2(t))||x-x_t^*||^2, \label{ref1}
		\end{eqnarray}	
	based on the fact that	if $z \in \Delta_{1}$,  then $x \in \Delta_{11}$.
		
	In addition,  there exists a positive constant $b_3$ such that for all $x\in \Delta_{11}$ and  $y_i\in\Delta_{12}\triangleq  \{y_i\in\mathbb{R}^N|||y_i||\leq 2(||z(0)|| +c') +c_x \}$,
	\begin{eqnarray}
	| \nabla_{x_{i}} f_{i}(x)-\nabla_{x_{i}} f_{i}(y_i)|\leq b_3 ||x-y_i||. \label{Lipschitz}
	\end{eqnarray}

		Since	$\sum_{i=1}^N||x-y_i||^2=\sum_{i=1}^N||\hat{Y}_i||^2$, it can be verified that if $z \in \Delta_{1}$,   $y_i \in \Delta_{12}$.
	
	Thus, for all $z\in\Delta_1$,
	\begin{eqnarray}
	||\Phi_{\delta\varepsilon}(Y)-\Phi_{\delta\varepsilon}(x)||^2\leq (b_3^2+b_2^2\varepsilon^2(t))\sum_{i=1}^N||\hat{Y}_i||^2.
	\end{eqnarray}

	It follows from (\ref{ref1}) that  for all $z\in\Delta_1$,
	\begin{eqnarray}
		& &||\Phi_{\delta\varepsilon}(Y)-\Phi_{\delta\varepsilon}(x_t^*)||^2\notag\\
		&\leq& (Nb_1'^2+Nb_2^2\varepsilon^2(t)+\delta^2(t))||x-x_t^*||^2\notag \\
		& &+(b_3^2+b_2^2\varepsilon^2(t))\sum_{i=1}^N||\hat{Y}_i||^2.
		\end{eqnarray}

		Thus, based on (\ref{r12}),  for all $z\in\Delta_1$,
			\begin{eqnarray}
		& &||\mathcal{P}_{\Omega}[x-\gamma(t)\Phi_{\delta\varepsilon}(Y)]-x_t^*||^2\notag\\
		&\leq& (1-u_1(t)) ||x-x_t^*||^2
		+ u_2(t)\sum_{i=1}^N||\hat{Y}_i||^2,
		\end{eqnarray}
		where $u_1(t)=-\gamma^2(t)(Nb_1'^2+Nb_2^2\varepsilon^2(t)+\delta^2(t)+1)+2\gamma(t)\delta(t)$, $u_2(t)= (\gamma^2(t)+1)(b_3^2+b_2^2\varepsilon^2(t))$. 
		
		Furthermore,  for all $z\in\Delta_1$,
			\begin{eqnarray}
		& &||\mathcal{P}_{\Omega}[x-\gamma(t)\Phi_{\delta\varepsilon}(Y)]-x||^2\notag\\
		&\leq& (2-u_1(t)) ||x-x_t^*||^2
		+ u_2(t)\sum_{i=1}^N||\hat{Y}_i||^2.
		\end{eqnarray}


Similar to the proof in Theorem \ref{theorem1}, by contradiction, suppose that there exists a time instant $T$ such that $z(t)\in\Delta_1$ for all $0\leq t<T$ and $z(T)\notin\Delta_1$ .
	
	 Taking the derivative of $V(t,z)$, we have for all $0\leq t<T$,
	\begin{align}
	\dot{V}\overset{a.e.}=&(x-x_t^*)^T(\dot{x}-\dot{x}_t^*)+\sum_{i=1}^N\hat{Y}_i^T\dot{\hat{Y}}_i\notag \\
	\overset{a.e.}=&\varsigma(t)(x-x_t^*)^T(\mathcal{P}_{\Omega}[x-\gamma(t)\Phi_{\delta\varepsilon}(Y)]-x_t^*)\notag \\ &-\varsigma(t)(x-x_t^*)^T(x-x_t^*)
	-(x-x_t^*)^T\dot{x}_t^* \notag\\
	 &-w(t)\sum_{i=1}^N\hat{Y}_i^T(L_i+B_i)\hat{Y}_i-\sum_{i=1}^N\hat{Y}_i^T\dot{x}_i\mathbf{1}_{N-1}\notag \\
	\overset{a.e.}\leq&-\frac{u_1(t)\varsigma(t)}{2}||x-x_t^*||^2-(x-x_t^*)^T\dot{x}_t^*\notag \\
	 &-(w(t)\lambda_{\text{min}}-\frac{u_2(t)\varsigma(t)}{2}\notag\\
	 &-\frac{\varsigma (t)(2-u_1(t))(N-1)^2}{u_1(t)})\sum_{i=1}^N||\hat{Y}_i||^2 \notag\\
	 &+\frac{u_1(t)\varsigma(t)}{4(2-u_1(t))}||\mathcal{P}_{\Omega}[x-\gamma(t)\Phi_{\delta\varepsilon}(Y)]-x||^2  \notag\\
	\overset{a.e.}\leq&-\frac{u_1(t)\varsigma(t)}{4}||x-x_t^*||^2-(x-x_t^*)^T\dot{x}_t^*\notag\\
	&-\theta(t)\sum_{i=1}^N||\hat{Y}_i||^2, \notag
	\end{align}
	where $\theta(t)=w(t)\lambda_{\text{min}}-\frac{u_2(t)\varsigma(t)}{2}-\frac{\varsigma (t)(2-u_1(t))(N-1)^2}{u_1(t)}-\frac{u_1(t)u_2(t)\varsigma(t)}{4(2-u_1(t))}$.
	
	Since $\theta(t)\geq \frac{1}{4}u_1(t)\varsigma(t)$, for all $0\leq t<T$,
	\begin{eqnarray}
	\dot{V}&\overset{a.e.}\leq& -\frac{u_1(t)\varsigma(t)}{2}V+\sqrt{2V}\left\Vert\dot{x}_t^*\right\Vert.
	\end{eqnarray}
	
	The following analysis is similar to that in Theorem \ref{theorem1} by defining $\phi=\sqrt{2V}=||z||$ and  $E_{u_1\varsigma}(t)\triangleq \int_{s=0}^t \frac{u_1(s)\varsigma(s)}{4}ds$. Thus, it is omitted here.
	
	As a byproduct, we can get that for all $t\geq 0$,
	\begin{eqnarray}
||z(t)||\leq e^{-E_{u_1\varsigma}(t)}(||z(0)||+c_{px}\int_{s=0}^{t} r_2(s)e^{E_{u_1\varsigma}(s)}ds), \label{lip}
	\end{eqnarray}
	which will be used in the following development.

		\color{black}
\end{proof}

\color{black}
\begin{remark} \label{rem1}
	The solutions to the conditions in Theorem \ref{t1} exist. For example,  $\varsigma(t)=(1+t)^6$, $\delta(t)=\sqrt{Nb_1'^2+1}(1+t)^{-\frac{1}{2}}$,  $\varepsilon(t)=\frac{\sqrt{Nb_1'^2+1}}{b_2}(1+t)^{1.2}$, $\gamma(t)=\frac{\delta (t)}{Nb_1'^2+Nb_2^2\varepsilon^2(t)+\delta^2(t)+1}$. The condition on $\theta(t)$ is easy to be satisfied and thus is omitted here and in the following analysis.
	\end{remark}
 
 \begin{remark}
 In this work, the "distributed" algorithm means that the players does not require  real-time strategy/state information exchanges among nonneighboring players. The common parameters can be set up before playing the game.
 \end{remark}

We can further get the following conclusion.
\begin{proposition}
	Under Assumptions \ref{fassumption}--\ref{graph} and assuming that $\nabla_{x_i}f_i(x)$ is globally Lipschitz continuous, the distributed algorithm given in (\ref{dynamics1})  asymptotically converges to the least-norm variational equilibrium of GNEP (\ref{P1}), provided that the conditions in (\ref{b12}), (\ref{C1}) and   (\ref{C2}) hold. \label{theorem3}
\end{proposition}
\color{black}

\color{black}
\begin{remark}
	It can be verified that  $\varsigma(t)=(1+t)^6$, $\delta(t)= (1+t)^{-\frac{1}{2}}$,  $\varepsilon(t)= (1+t)^{1.2}$, $\gamma(t)=\frac{\delta (t)}{Nb_1'^2+Nb_2^2\varepsilon^2(t)+\delta^2(t)+1}$. satisfy the conditions in  Proposition \ref{theorem3}.
	\end{remark}


\subsection{Convergence Rate Analysis}

In this section, we investigate the convergence rate of the algorithm in (\ref{dynamics1}). Note that the convergence of the algorithm relies on two parts: the convergence of the state $x$ to the trajectory $x_t^*$ and the convergence of the trajectory  $x_t^*$  to the least-norm solution $x^*$. The former one can be described by  (\ref{lip}) and the latter one is generally difficult to be obtained in a regularized algorithm. Nevertheless, according to (\ref{lim12}), we can get the following conclusion.
 		
		\begin{proposition}
			Suppose that the conditions in Theorem \ref{t1} hold. Then,  
					\begin{eqnarray} 
||x-x_t^*||\leq e^{-E_{u_1\varsigma}(t)}(||z(0)||\notag \\
+c_{px}\int_{s=0}^{t} r_2(s)e^{E_{u_1\varsigma}(s)}ds). \label{s}
			\end{eqnarray} 
			Furthermore, if $||x^*||\neq 0$,
			\begin{eqnarray} 
			||x||-||x^*||&\leq \frac{C_3}{||x^*||\delta(t)}\sqrt{\frac{C_2}{\varepsilon(t)}}+e^{-E_{u_1\varsigma}(t)}(||z(0)|| \notag \\
			&+c_{px}\int_{s=0}^{t} r_2(s)e^{E_{u_1\varsigma}(s)}ds). \label{c1}
			\end{eqnarray} 
			If  $||x^*||=0$,
		    \begin{eqnarray} 
			||x|| &\leq (\frac{C_3}{\delta(t)}\sqrt{\frac{C_2}{\varepsilon(t)}})^{\frac{1}{2}}+e^{-E_{u_1\varsigma}(t)}(||z(0)|| \notag \\
			&+c_{px}\int_{s=0}^{t} r_2(s)e^{E_{u_1\varsigma}(s)}ds), \label{c2}
			\end{eqnarray} 
			where $C_2>0$ and $C_3>0$ were defined in Lemma \ref{relationship}, and $c_{px}=\min\{c_{p}, c_x\}$ with $c_{p}$ and $c_x$ being defined in Lemma \ref{xbound}.
	\end{proposition}

\begin{proof}
 (\ref{s}) is a direct result of (\ref{lip}). (\ref{c1}) is obtained by dividing (\ref{lim12}) by $||x_t^*||+||x^*||$ and adding  (\ref{s}).   (\ref{c2}) is obtained  by taking the square root of (\ref{lim12}) and adding  (\ref{s}). 
	\end{proof}

\color{black}

			
				\color{black}
				
				\subsection{The Unconstrained Monotone Game}
				
				In the above sections, we design a distributed algorithm for monotone generalized noncooperative games with set constraints and shared affine  inequality constraints. In this session, we study the unconstrained monotone game described as follows:
				\begin{equation}
					\min_{x_{i}\in\mathbb{R}}\text{ }f_{i}(x_{i},x_{-i}). \label{P3}%
				\end{equation}
	
		 Note that it cannot be viewed as a special case of  GNEP (\ref{P1}) since the boundedness assumption of the strategy space is not satisfied. 
		 
		 We consider the following distributed updating law
		 \begin{align}
		 	&\dot{x}_i=-(\nabla_{x_{i}} f_{i}(y_i)+\delta(t)x_i), \notag \\
		 	&\dot{y}_{ij}=-w(t)\sum_{k=1}^{N}a_{ik}(y_{ij}-y_{kj}),j\in\mathcal{V}\backslash i, \label{dynamics12}
		 \end{align}
		 where $y_{ij}\in\mathbb{R}$,  $y_{ii}=x_i$, $y_i=[y_{ij}]_{j\in\mathcal{V}}$, $\delta(t)>0$ and $ w(t)>0$ are time-varying, positive and smooth parameters with bounded initial values.
				
				In this section, we assume that Assumptions	\ref{fassumption}-\ref{monot} and \ref{graph} hold. In addition, the following assumption is required to guarantee the existence of a Nash equilibrium of the game.
				
				\color{black}
				\begin{assumption} 
					The Nash equilibrium set $\mathcal{S}$ of  game (\ref{P3}) is nonempty. \label{nonempty}
				\end{assumption}
			
			At each $t\geq 0$, define the following unconstrained game
			\begin{equation}
				\min_{x_{i}}\text{ }f_{i}(x_{i},x_{-i})+\frac{\delta(t)}{2}x_i^2. \label{P2}%
			\end{equation}
			
			According to Assumption \ref{monot}, the regularized map $\Phi_{\delta}(x)= F(x)+\delta(t)x$ is strongly monotone at each $t\geq 0$. According to Theorem 3 of \cite{scutari2014real}, the unconstrained game in (\ref{P2}) has a unique Nash equilibrium at each $t\geq 0$, denoted by $x_t^*$. 
			
			Let $x^*$ be an element in $\mathcal{S}$ with the least norm. The following lemma shows the relationship between $x_t^*$ and $x^*$.
			\begin{lemma} 
				Suppose that  Assumptions	\ref{fassumption}--\ref{monot} and \ref{nonempty} hold  and $\lim_{t\rightarrow\infty} \delta(t)=0$. Then, $\lim_{t\rightarrow\infty}x_t^*=x^*$ and $||x_t^*||\leq||x^*||$.  \label{tends}
			\end{lemma}
			\begin{proof}
				The proof follows Lemma 1 of \cite{bruck1974strongly} (or Lemma 4 of \cite{cominetti2008strong}), where we made some minor revisions.
				
				 According to the monotonicity, $F\left(x_{t}^{*}\right)^{T}\left(x_{t}^{*}-x^{*}\right) \geq 0$ at any $t \geq 0,$ which implies that $\left(-\delta(t) x_{t}^{*}\right)^{T}\left(x_{t}^{*}-x^{*}\right) \geq 0 .$ Thus, $\left\Vert {}x_{t}^{*}-x^{*}\right\Vert  ^{2} \leq\left\Vert  x^{*}\right\Vert  ^{2}-\left\Vert {}x_{t}^{*}\right\Vert  ^{2}$
				and $\left\Vert {}x_{t}^{*}\right\Vert   \leq\left\Vert {}x^{*}\right\Vert  $.
				
				Let $0<\vartheta(t)<\delta(t).$ Then, $(F\left(\tilde{x}_{t}^{*}\right)-F\left(x_{t}^{*}\right))^{T}(\tilde{x}_{t}^{*} -x_{t}^{*}) \geq$
				$0,$ where $\tilde{x}_{t}^{*}$ is the unique solution to $0=F(x)+\vartheta(t) x$ Thus, $\left(\delta(t) x_{t}^{*}-\vartheta(t) \tilde{x}_{t}^{*}\right)^{T}\left(\tilde{x}_{t}^{*}-x_{t}^{*}\right) \geq 0,$ which implies that
				$(\vartheta(t)-\delta(t)) x_{t}^{* T}\left(x_{t}^{*}-\tilde{x}_{t}^{*}\right) \geq\left\Vert {}x_{t}^{*}-\tilde{x}_{t}^{*}\right\Vert  ^{2} \geq 0 .$ Furthermore,
				$x_{t}^{* T}\left(x_{t}^{*}-\tilde{x}_{t}^{*}\right) \leq 0$ and it follows that $\left\Vert {}x_{t}^{*}\right\Vert  \leq\left\Vert {}\tilde{x}_{t}^{*}\right\Vert.$ Thus, $x_{t}^{*}$ is
				monotonically nondecreasing. According to the boundedness and the monotone convergence theorem, $\lim _{t \rightarrow \infty} x_{t}^{*}$ exists.
				
				Let $\bar{x}=\lim _{t \rightarrow \infty} x_{t}^{*}$. Since $\lim _{t \rightarrow \infty} \delta(t)=0$ and
				$\lim_{t \rightarrow \infty} \delta(t) x_{t}^{*}=0$,
				$\lim_{t \rightarrow \infty} F(x_{t}^{*})=0$. According to Assumption 1 and the first-order condition, $\mathcal{S} =\{x \in\mathbb{R}^{N} | F(x)=0 \}.$ Then, according to the closeness of $F(x),$ we have $F(\bar{x})=\lim _{t \rightarrow \infty} F(x_{t}^{*})=0,$ i.e., $\bar{x} \in \mathcal{S} .$ Since
				$\left\Vert  x_{t}^{*}\right\Vert   \leq\left\Vert  x^{*}\right\Vert ,$ we have $\left\Vert \bar{x}\right\Vert \leq\left\Vert  x^{*}\right\Vert  .$ Thus, $\bar{x}=x^{*}$ and $\lim _{t \rightarrow \infty} x_{t}^{*}=x^{*}.$ The convergence is strong since the considered space is finite-dimensional.
			\end{proof}
			
			\begin{lemma}
			 Under the conditions in Lemma \ref{tends}, (a) $\dot{x}_t^*$ exists almost everywhere; (b)	$||\dot{x}_t^*||\leq \frac{c_x|\dot{\delta}(t)|}{\delta(t)}$. \label{xbound2}
			\end{lemma}
			
			\begin{proof}
				The proof is given in \cite{attouch1996dynamical} and \cite{boct2020inducing}. It can be viewed as a special case of Lemma \ref{xbound} and thus is omitted.
			\end{proof}
			
			The main result of this section can be described as follows:
			
			\color{black}
			\begin{theorem}
				Under Assumptions \ref{fassumption}--\ref{monot} and \ref{graph}--\ref{nonempty}, the algorithm given in (\ref{dynamics12})  asymptotically converges to the least-norm variational equilibrium in $\mathcal{S}$, provided that the following conditions hold:
				\begin{eqnarray}
				 \lim_{t\rightarrow\infty} \delta(t)=0, 
					 \lim_{t\rightarrow\infty}\int_{s=0}^t\delta(s)ds=\infty, 
					\lim_{t\rightarrow\infty} \frac{|\dot{\delta}(t)|}{\delta^2(t)}=0,  \notag\\ 
					\int_{s=0}^t\frac{|\dot{\delta}(s)|}{\delta(s)}e^{\int_{p=0}^s\delta(p)dp}ds \text{ is bounded or } \notag \\
					\lim_{t\rightarrow\infty}\int_{s=0}^t\frac{|\dot{\delta}(s)|}{\delta(s)}e^{\int_{p=0}^s\delta(p)dp}ds=\infty; \\
					w(t)\lambda_{\text{min}}-\frac{l_1^2c_1}{2\delta(t)}-l_1\sqrt{N-1} -\frac{(N-1)c_2(l_2+\delta(t))^2}{2\delta(t)} \notag \\
					\geq\delta(t)(1-\frac{N}{2c_1}-\frac{N}{2c_2})> 0,
				\end{eqnarray}
				where $c_1$ and $c_2$ are arbitrary positive constants, $l_1$ and $l_2$ are some local Lipschitz constants, and $\lambda_{\min}$ is the minimum value among the eigenvalues of $L_i+B_i$, $i\in\mathcal{V}$. \label{t2}
			\end{theorem}
		
		\begin{proof}
			The idea of the proof is similar to that in Theorem  \ref{t1}, and is thus omitted.
			\end{proof}

	\section{Simulation}\label{S5}

	\subsection{Algorithm Verification} \label{S6}
	
	To verify the effectiveness of the proposed algorithm, we consider a 5-person noncooperative game with the following objective functions: 
	\begin{eqnarray}
	&&f_1(x)=x_1^2-x_1x_2-x_1x_5, \notag\\
    &&f_2(x)=1.5x_2^2-x_1x_2-x_2x_3-x_2x_4, \notag\\
     &&f_3(x)=0.5x_3^2-x_2x_3, \notag\\
     &&f_4(x)=0.5x_4^2-x_2x_4, \notag\\
     &&f_5(x)=0.5x_5^2-x_1x_5. 
	\end{eqnarray}
	
  	Each player $i$ is subject to a local constraint  $\Omega_i=\{x_i\in\mathbb{R} | -i\leq x_i\leq i\}$ and a shared constraint $g(x)=x_1+x_2+x_3+x_4+x_5+1\leq 0$. 
  	The communication graph for the 5 players is shown in Fig. \ref{fig:graph}.

It can be calculated that the least-norm variational equilibrium of the game is $[-0.2,-0.2,-0.2,-0.2,-0.2]^T$.
	
	\begin{figure} [htp]
		\centering
		\includegraphics[width=0.8\linewidth]{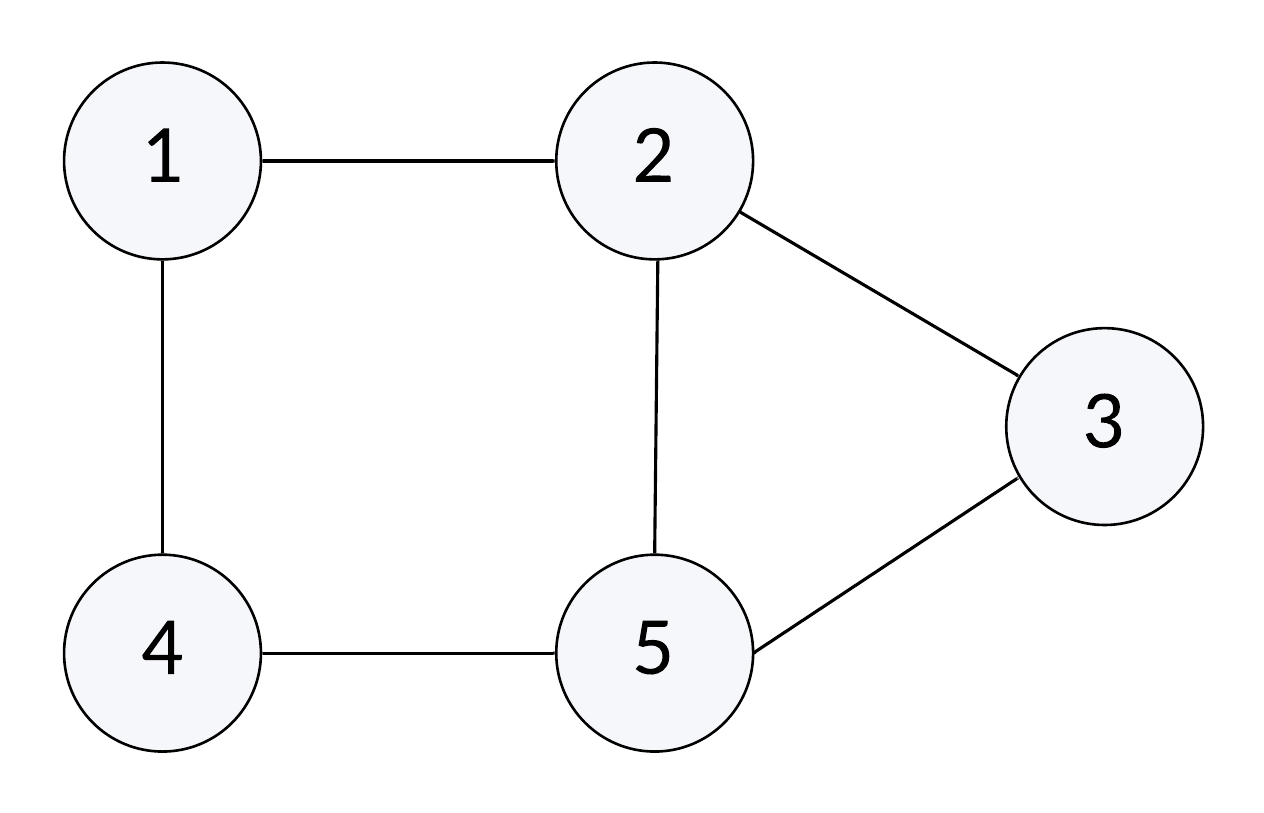}
		\caption{The communication graph of the network in Section \ref{S6}.}
		\label{fig:graph}
	\end{figure}

\begin{figure}[htpb]
	\centering
	\includegraphics[width=1\linewidth]{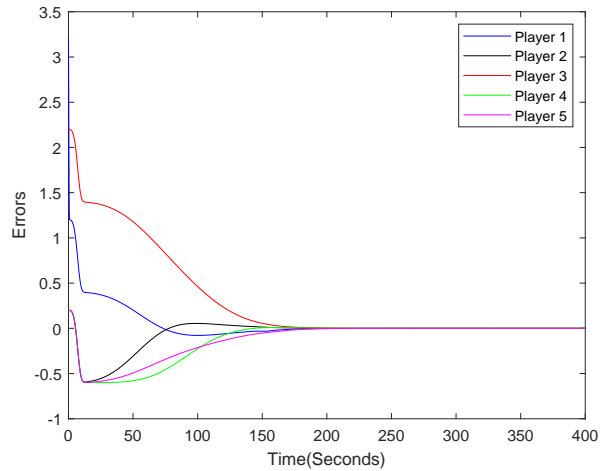}
	\caption{The error with the least-norm variational equilibrium.}
	\label{fig:fig1}
\end{figure}

	The initial states of the players are $x(0)=[3,0,2,0,0]^T$ and $y_{ij}=0, i\neq j$.
	Let (\ref{dynamics1}) be the updating law with $\delta(t)=0.1(1+t)^{-0.5}$, $\varepsilon(t)=20(1+t)^{1.2}$, $\gamma(t)=\delta(t)/(126+ \delta(t)^2+125\varepsilon(t)^2)$, $\varsigma(t)=(1+t)^5$, $w(t)=500+500(1+t)^{9}$. The simulation result is shown in Fig.  \ref{fig:fig1}. It can be seen that the algorithm converges to the least norm Nash equilibrium, which verifies Theorem  \ref{t1}.
	
	To further verify the role of the regularization term in terms of the least-norm property, we consider the case where the shared constraint $g(x)=0$.	The least-norm variational equilibrium in this case is $[0,0,0,0,0]^T$. Let (\ref{dynamics1}) be the updating law with (a) $\delta(t)=0$, $\varepsilon(t)=0$, $\gamma(t)=0.1$, $\varsigma(t)=1$,  $w(t)=500$ for the unregularized case; (b) $\delta(t)=0.1(1+t)^{-0.5}$, $\varepsilon(t)=20(1+t)^{1.2}$, $\gamma(t)=\delta(t)/(126+ \delta(t)^2)$, $\varsigma(t)=(1+t)^5$, $w(t)=500+500(1+t)^{9}$ for the regularized case. The simulation results for both cases are shown in Fig. 	\ref{Fig.main}. It can be seen that the steady states are different. The regularized algorithm converges to the least-norm solution while the unregularized algorithm doesn't.

\begin{figure}[htpb]
	\centering 
	\subfigure[The unregularized case]{
		\label{Fig.sub.1}
		\includegraphics[width=0.5\textwidth]{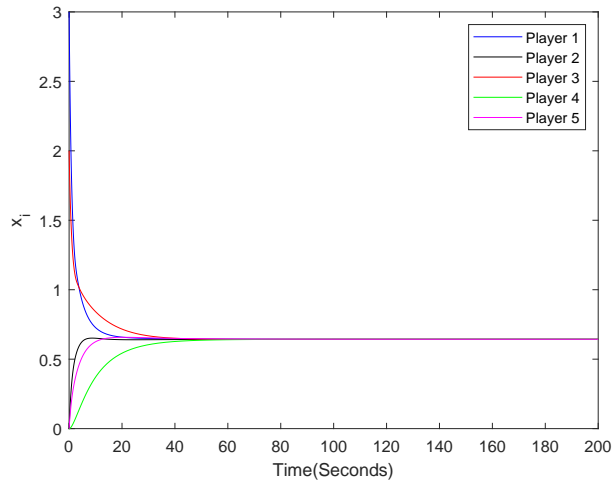}}
	\subfigure[The regularized case]{
		\label{Fig.sub.2}
		\includegraphics[width=0.5\textwidth]{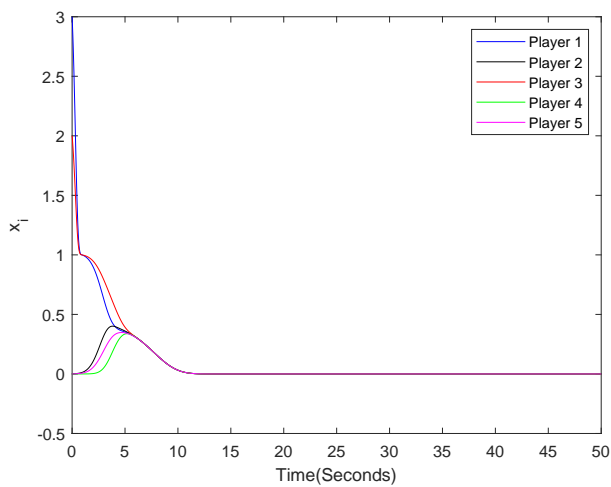}}
      \caption{The role of the regularization term.}
	\label{Fig.main}
\end{figure}
	\color{black}

		\subsection{Applications to Connectivity Control Problems in Robotics} \label{S7}
		
In this section, we consider a connectivity control issue studied in \cite{stankovic2011distributed}. The robot team is consisting of 8 robots with the communication graph described in Fig. 	\ref{fig:graph7}. 

Each robot aims to
achieve a compromise between a local goal $\mathcal{J}_i(x_i)$ and a collective goal  $\mathcal{C}_i(x)$. The objective function for robot $i$  can be written as $f_i(x)=a_i\mathcal{J}_i(x_i)+(1-a_i)\mathcal{C}_i(x)$ with $0\leq a_i\leq 1$ being the coefficient that describes how ``selfish'' robot $i$ is. If $a_i=1$, then robot $i$ is totally selfish and does not care about the collective goal. If $a_i=0$, then robot $i$ is collaborative and does not care about its private objective.

\begin{figure}
	\centering
	\includegraphics[width=0.6\linewidth]{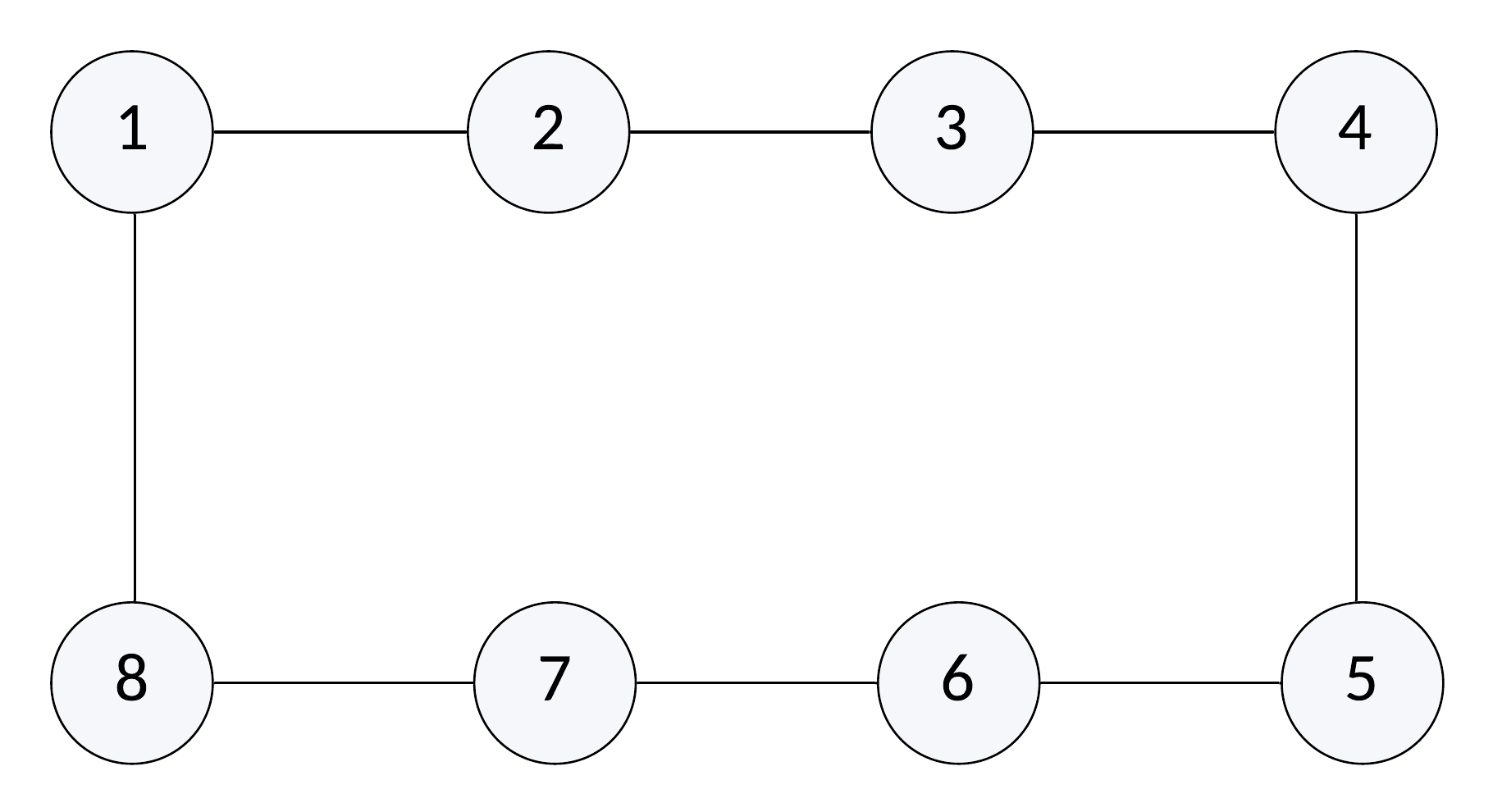}
	\caption{The communication graph of the network in Section \ref{S7}.}
	\label{fig:graph7}
\end{figure}

In the simulation, we set $\mathcal{J}_i(x_i)=(x_i-i)^2$, which indicates a desired distance towards a point, and $\mathcal{C}_i(x)=\sum_{j\in\mathcal{V}_i}||x_i-x_j||^2$, which represents the objective to keep good connections with the neighboring agents. It can be verified that $F(x)=[\nabla_{x_i}f_i(x)]_{i\in\{1,\cdots, 8\}}$ is monotone. Furthermore, if $a_i\neq 0$ for all $i\in\{1,\cdots, 8\}$, $F(x)$ is strongly monotone. Elsewise, it is merely monotone. Furthermore, if $a_i=0$ for all $i\in\{1,\cdots, 5\}$, the problem becomes a  consensus problem.

		Suppose that the  robots' optimal positions satisfy $0\leq x_i^*\leq 10$ and $|x_i^*-x_j^*|\leq 1$, $i\in{1,\cdots,8}$, $j\in\mathcal{V}_i$, where the latter constraints indicate  some  connectivity-related constraints at optimal solutions.
			
		\begin{remark}
		Note that by the proposed method, we cannot guarantee that the distance between two robots is always no larger than the connectivity constraint distance during the motion, but can only guarantee that the optimal solution satisfies this requirement. In control and robotics, potential function based methods are often used to achieve some tasks with state constraints during motion (e.g., collision avoidance and connectivity preservation). In addition, in optimization and game theory, barrier function based methods can also guarantee that the constraints are always satisfied with the iterations. In future, we may consider to apply these methods to guarantee that the constraints can be satisfied all the time. Specifically, barrier function based methods are closely related with  penalty function based methods adopted in this work. They have a similar spirit, where a penalty function focuses on the constraint violation outside of the constraint set while a barrier function focuses on the  constraint guarantee inside of the constraint set. Thus, an open question is that if we replace the penalty function used in this work with a barrier function, what will happen to the algorithm? Usually a barrier function based method requires the constraints to be satisfied initially, which may be a disadvantage compared with a penalty function based method. In practice, the two kinds of methods can be applied according to different application scenarios.
		\end{remark}
%
		
	Let $a_1=a_2=a_3=0.2$, $a_4=a_5=0,$  and $a_6=a_7=a_8=0.6$. It can be calculated that the least-norm variational equilibrium of the game is $[5.4018,4.7259,4.7316,5.1702,$ $5.6088,6.0473,6.5569,6.4018]^T$. 
	
		The initial states of the players are $x(0)=[0,4,10,6,-1,0,12,0]^T$ and $y_{ij}=0, i\neq j$.
	Let (\ref{dynamics1}) be the updating law with $\delta(t)=0.01(1+t)^{-0.5}$, $\varepsilon(t)=2(1+t)^{1.2}$, $\gamma(t)=\delta(t)/(201+ \delta(t)^2+200\varepsilon(t)^2)$, $\varsigma(t)=2(1+t)^5$, $w(t)=500+500(1+t)^{10}$. The simulation results are shown in Figs. \ref{fig:fig4} and \ref{fig:fig5}. It can be seen that the algorithm converges to the least-norm variational equilibrium.
	
\begin{figure}
	\centering
	\includegraphics[width=1\linewidth]{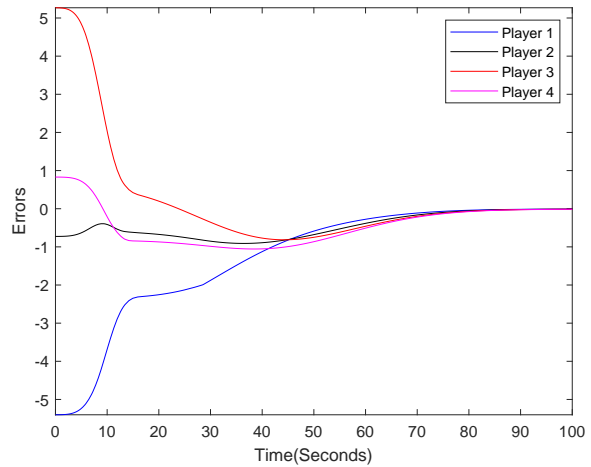}
	\caption{The errors of Players 1-4.}
	\label{fig:fig4}
\end{figure}

\begin{figure}
	\centering
	\includegraphics[width=1\linewidth]{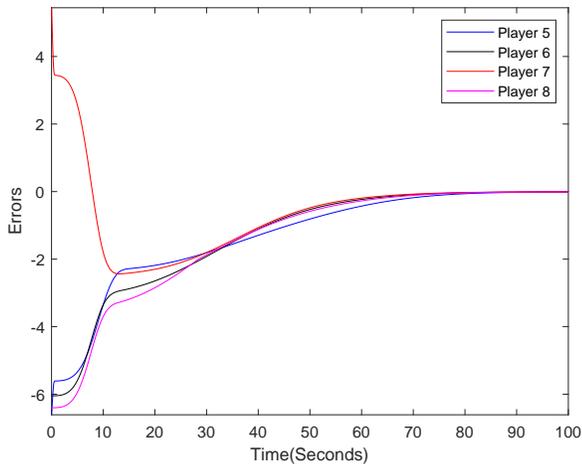}
	\caption{The errors of Players 5-8.}
	\label{fig:fig5}
\end{figure}

	\color{black}
	\section{Conclusions}
	
	In this work, we propose a novel distributed regularized penalty method to solve the Nash equilibrium seeking problem of a monotone generalized noncooperative game, by using a time-varying penalty parameter and a time-varying regularization term. A framework of the regularized penalty method is proposed. It is proven that the solution to the regularized penalized problem asymptotically converges to the least-norm variational equilibrium of the original generalized noncooperative game. Furthermore, the proposed projected gradient dynamics achieve an asymptotic convergence to the solution to the regularized penalized problem. Therefore, an asymptotically convergent algorithm is obtained. Simulation results show the effectiveness and efficiency of the proposed algorithm. 
				
		\color{black}

\bibliographystyle{IEEEtran}
\bibliography{bib}

	~\\
	~\\

		\parpic{\includegraphics[width=1in,clip,keepaspectratio]{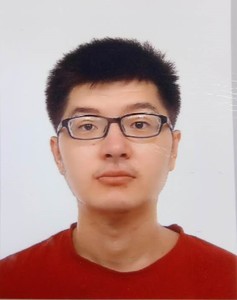}}
		\noindent {\bf Chao Sun}  received his B.Eng degree from University of Science and Technology of China in 2013. He obtained the Ph.D. degree from School of Electrical and Electronics Engineering, Nanyang Technological University, Singapore in 2018. He was a Wallenberg-NTU Presidential Postdoctoral Fellow at Nanyang Technological University from June 2019 to August 2020. He is now working as a research fellow at Nanyang Technological University. His research interests include distributed optimization and game theory.

			~\\
			~\\

		\parpic{\includegraphics[width=1in,clip,keepaspectratio]{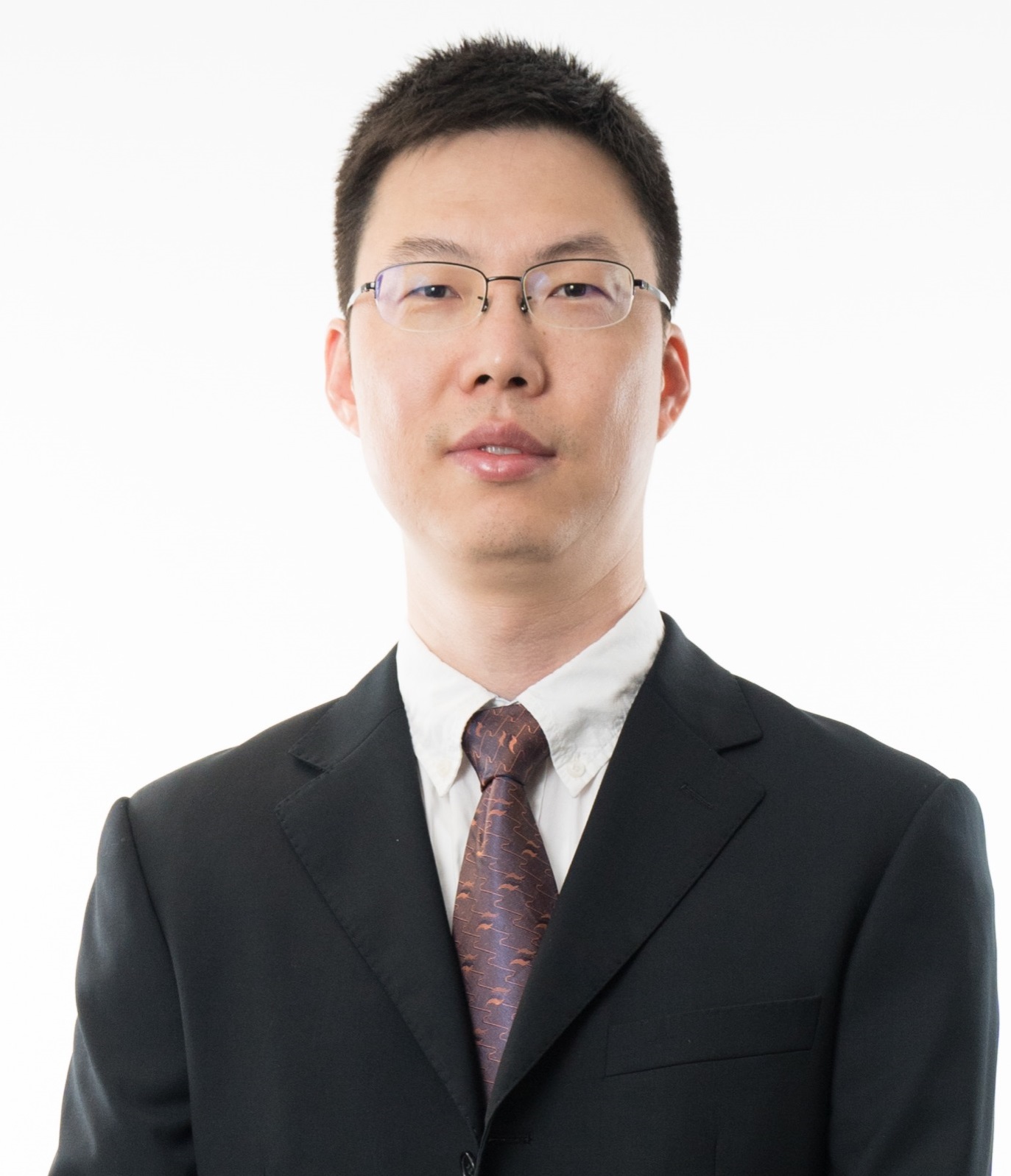}}
		\noindent {\bf Guoqiang Hu} joined the School of Electrical and Electronic Engineering at Nanyang Technological University, Singapore in 2011, and is currently a tenured Associate Professor and the Director of the Centre for System Intelligence and Efficiency. He received Ph.D. in Mechanical Engineering from the University of Florida in 2007. He works on distributed control, distributed optimization and game theory, with applications to multi-robot systems and smart city systems. He was a recipient of the Best Paper in Automation Award in the 14th IEEE International Conference on Information and Automation, and a recipient of the Best Paper Award (Guan Zhao-Zhi Award) in the 36th Chinese Control Conference. He serves as Associate Editor for IEEE Transactions on Control Systems Technology and IEEE Transactions on Automatic Control.

\end{document}